\documentclass[12pt]{article}
\usepackage{amsfonts}

\usepackage{mathrsfs}
\usepackage{srcltx}
\usepackage{xcolor}
\usepackage{amsmath}
\usepackage{amsthm}
\usepackage{amstext}
\usepackage{amsopn}
\usepackage{float}
\usepackage{subcaption}
\usepackage{tikz}
\usepackage{microtype}
\usepackage{graphicx}
\usepackage{bm}

\newtheorem{theorem}{Theorem}[section]
\newtheorem{lemma}[theorem]{Lemma}
\newtheorem{proposition}[theorem]{Proposition}
\newtheorem{corollary}[theorem]{Corollary}

\theoremstyle{definition}

\theoremstyle{remark}
\newtheorem{remark}[theorem]{Remark}

\numberwithin{equation}{section}

\newcommand{\ba}{\begin{array}}
\newcommand{\ea}{\end{array}}
\newcommand{\f}{\frac}

\newcommand{\la}{\lambda}

\newcommand{\ds}{\displaystyle}

\begin{document}
\date{}
\title{ \bf\large{Global dynamics of a two-species competition patch model in a Y-shaped river network}}
\author{Weifang Yan\textsuperscript{1} and\;Shanshan Chen\textsuperscript{2}\footnote{Corresponding Author, Email: chenss@hit.edu.cn; chenshanshan221@126.com}\ \
 \\
{\small \textsuperscript{1} School of Mathematics and Statistics Science, Ludong University,\hfill{\ }}\\
\ \ {\small Yantai, Shandong, 264001, P.R.China.\hfill{\ }}\\
{\small \textsuperscript{2} Department of Mathematics, Harbin Institute of Technology,\hfill{\ }}\\
\ \ {\small Weihai, Shandong, 264209, P.R.China.\hfill{\ }}\\
}
\maketitle

\begin{abstract}
In this paper, we investigate a two-species Lotka-Volterra
competition patch model in a Y-shaped river network, where
the two species are assumed to be
identical except for their random and directed
movements. We show that competition exclusion can occur under certain conditions, i.e., one of the semi-trivial equilibria
is globally asymptotically stable. Specifically, if the random dispersal rates of the two species are equal, the species with a smaller drift rate will drive the other species to extinction, which suggests that smaller
drift rates are favored.
\\[2mm]
\end{abstract}

\section{Introduction}
The species in rivers are subject to unidirectional flow, which
washes them downstream. One basic question (Q1) in river ecology is  \lq\lq drift paradox'' \cite{2-Muller-1954}: how species can persist in rivers with the flow-induced washout?
In the framework of reaction-diffusion-advection (RDA) models, Speirs and Gurney \cite{1-Speirs-2001} firstly showed that
the species can persist when the drift rate induced by the unidirectional flow is relatively slow, and the river is long enough. Similar results were obtained in
\cite{3-Lutscher-2014,4-Lou-Zhou-2015,5-Vasilyeva-Lutscher} and \cite{6-Chen-Shi-Shuai-Wu-2023} for RDA models and patch models with different boundary
conditions at the downstream end, respectively. In addition,
other factors, such as seasonal environments, Allee effect, etc., were also considered in persistence of stream-dwelling organisms, see \cite{7-Jin-Lewis-2011,8-Jin-Lewis-2012,9-Wang-Shi-2019,10-Wang-Shi-2019} and references therein.

Another basic question (Q2) in river ecology is what kind of strategies has competitive advantages. This question was first studied in spatially heterogeneous non-advective environments. It was shown in \cite{16-Dockery-1998,17-Hastings-1983} that, if the two species are identical except for their diffusion rates, the slower diffuser can wins the competition. In addition, the  global dynamics for the case of weak competition was studied in \cite{18-He-Ni-2016,19-Lam-Ni-2012,20-Lou-2006} and references therein.
The interaction between two competing species in rivers can be described by the following RDA model:
\begin{equation}\label{adv3}
\begin{cases}
\ds u_t=d_1 u_{x x}-q_1 u_x+ u\left(r-u-v\right), & 0<x<L,\;\; t>0, \\
\ds v_t=d_2 v_{x x}-q_2 v_x+v\left(r-u-v\right), & 0<x<L,\;\; t>0, \\
d_1 u_x-q_1 u=d_2 v_x-q_2 v=0, &x=0,\;\; t>0,\\
d_1 u_x-q_1 u=-\beta q_1u,\;\;d_2 v_x-q_2 v=-\beta q_2 v, &x=L,\;\; t>0,\\
u(x, 0)=u_0(x) \geq(\not \equiv) 0, \quad v(x, 0)=v_0(x) \geq(\not \equiv) 0,
\end{cases}
\end{equation}
where $u$ and $v$ are the densities of  two species; $d_1,d_2$ and $q_1,q_2$ are the diffusion rates and the drift rates of the two species, respectively; $r$ is the intrinsic growth rate of the two species; and  $\beta$ represents the population loss at the downstream end.
Speirs and Gurney \cite{1-Speirs-2001} proposed the hostile boundary condition at the downstream end, corresponding to
	$\beta=\infty$. This represents a scenario where a stream flows into an ocean.
The free-flow boundary condition ($\beta=1$) at the downstream end represents
a stream flowing into a lake, while the no-flux boundary condition ($\beta=0$) corresponds to
an inland stream \cite{3-Lutscher-2014,5-Vasilyeva-Lutscher,10-Wang-Shi-2019}.
Specifically, the no-flux boundary condition ($\beta=0$) implies no loss of individuals at the downstream end, representing an inland stream, which refers to an endorheic river that disappears into basins or deserts.

Fixing $q_1=q_2=q$ and viewing diffusion rate as a strategy, the authors in \cite{3-Lutscher-2014} showed that the species with faster diffusion rate wins the competition for $\beta=1$, and this result was extended to the case $\beta\in[0,1]$ in \cite{4-Lou-Zhou-2015}. The case $\beta>1$ is complex, as even small diffusion rates can lead to competitive dominance \cite{11-Hao-Lam-Lou-2021,4-Lou-Zhou-2015}. Fixing $d_1=d_2=d$ and viewing drift rate as a strategy, the authors in \cite{12-Lou-Xiao-Zhou-2011,13-Zhou-2016} showed that the species with slower drift rate has competitive advantages. The effect of $d_1,q_1,d_2,q_2$ on the global dynamics of model \eqref{adv3} was investigated in \cite{13-Zhou-2016,14-Zhou-Zhao-2018}. For spatially heterogeneous environments (replace $r$ by $r(x)$ in model \eqref{adv3}), the global dynamics of \eqref{adv3} is complex \cite{27-Ge-Tang-2023,26-Ge-Tang-2021,28-Lou-Zhao-Zhou-2019,29-Zhao-Zhou-2016,30-Zhou-Zhao-2018}, and there may exist some intermediate diffusion rate which is evolutionarily stable \cite{31-Lam-Lou-2015}.
Moreover,
other competition models were also studied extensively, see \cite{21-Lou-Nie-Wang-2018,22-Yan-Li-Nie-2021,15-Yan-Nie-Zhou-2022} for
models with different intrinsic growth rates of the two species and \cite{23-Ma-Tang-2020,24-Tang-Chen-2020,25-Wang-Xu-Zhou-2024} for models with different boundary conditions at the upstream end.
The interaction among stream-dwelling organisms are complex. There are also extensive results on other
population models in rivers, including predator-prey models \cite{47-Lou-Nie-2022,48-Nie-Wang-Wu-2020,49-Tang-Chen-2022,50-Wang-Nie-2022}, benthic-drift models \cite{51-Huang-Jin-2016,52-Jiang-Huang-2019,53-Wang-Shi-2020} and so on.

The discrete patch model of \eqref{adv3} under the no-flux boundary condition
($\beta=0$) takes the following form:
\begin{equation}\label{App1}
	\begin{cases}
		\ds\frac{{\rm d} w_k}{{\rm d}t}=\ds\sum_{j=1}^{n}(d_1D_{kj}+q_1Q_{kj})w_j+w_k(r-w_k-z_k),&k=1,\cdots,n,\;\;t>0,\\
		\ds\frac{{\rm d} z_k}{{\rm d}t}=\ds\sum_{j=1}^{n}(d_2D_{kj}+q_2Q_{kj})z_j+z_k(r-w_k-z_k),&k=1,\cdots,n,\;\;t>0,\\
        \bm w(0)=\bm w_0\ge(\not\equiv)\bm0,\;\;\bm z(0)=\bm z_0\ge(\not\equiv)\bm0,
	\end{cases}
\end{equation}
where the patches are located in a stream without branches, as shown in Figure \ref{fig2}. In this paper, a branch refers to a tributary that diverges from a main river.
$w_k$ and $z_k$ are numbers of two competing species in patch $k$, respectively, $d_1$ and $d_2$ are random movement rates of the two species, and $q_1$ and $q_2$  are directed drift rates. Moreover, the $n\times n$ matrices
 $(D_{kj})$ and $(Q_{kj})$ represent the diffusion pattern and the directed movement pattern of individuals, respectively, where
 \begin{equation}\label{App2}
	D_{kj}=\begin{cases}
		1,	&k=j-1 \;\mbox{or}\; k=j+1,\\
		-2, &k=j=2,\cdots,n-1,\\
		-1, &k=j=1,n,\\
		0,  &\mbox{otherwise},
	\end{cases}
	\qquad
	Q_{kj}=\begin{cases}
		1,	&k=j-1,\\
		-1, &k=j=2,\cdots,n,\\
		0,  &\mbox{otherwise}.
	\end{cases}
\end{equation}
Taking species $u$ as an example, for $j\ne k$, $d_1D_{kj}\ge0$ represents the movement rate from
patch $j$ to patch $k$ driven by random dispersal, while $q_1Q_{kj}$ represents  the movement rate from
patch $j$ to patch $k$ driven by directed drift. Additionally, $-d_1D_{kk}\ge0$ represents the departure rate from patch $k$ due to random dispersal,
and $-d_1Q_{kk}\ge0$ represents the departure rate from patch $k$ due to directed drift.
The patch models can admit similar dynamics as RDA models like \eqref{adv3} under certain conditions, see \cite{37-Chen-Liu-Wu-2022,38-Chen-Liu-Wu-2023,6-Chen-Shi-Shuai-Wu-2023,34-Hamida-2017,32-Jiang-Lam-Lou-2020,33-Jiang-Lam-Lou-2021,36-Lou-2021,35-Noble-2015} and references therein.
\begin{figure}[htbp]
	\centering
	\includegraphics[width=10cm]{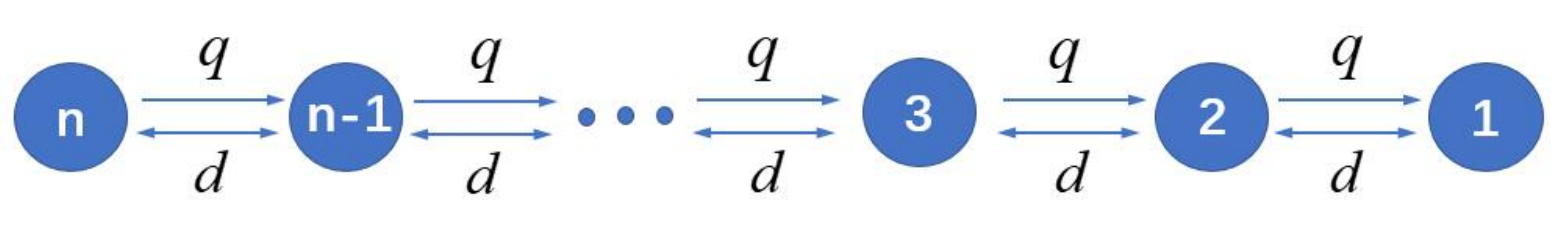}
	\caption{A stream with $n$ patches, where patch $n$ is the upstream end, and patch $1$ is the downstream end. Here $d$ is the random movement rate and $q$ is the directed drift rate for one species.} \label{fig2}
\end{figure}

Rivers may have complex topological structures, and ecologists have found that topology of river network can
affect population dynamics of  stream-dwelling organisms, see, e.g., \cite{S1-Cuddington-2002,S2-Fagan-2002,S3-Goldberg-2010,S4-Grant-2007,S5-Grant-2010}.
Several types of models were constructed to describe population dynamics of species in river networks, including integral-differential equations and RDA equations on  metric graphs \cite{44-Jin-Peng-2024,42-Jin-Peng-Shi-2019,39-Ramirez-2012,41-Sarhad-2014,40-Sarhad-2015,45-Vasilyeva-2019} and ordinary differential equations (e.g. patch models) \cite{S2-Fagan-2002,32-Jiang-Lam-Lou-2020,33-Jiang-Lam-Lou-2021,36-Lou-2021,S6-Casagrandi-2014}.

The above questions (Q1)-(Q2) were also concerned for species in river networks. For example, Vasilyeva \cite{45-Vasilyeva-2019} considered question (Q1) and studied the persistence of one species in a  Y-shaped river network (see Figure \ref{fig3}). To emphasize the effect of network geometry, it was assumed in \cite{45-Vasilyeva-2019} that the diffusion rate $d$, the drift rate $q$ and the intrinsic growth rate $r$ of the species are constant throughout the river network.
Then, for each river segment $i=1,2,3$, the population dynamics is modeled by the following RDA model:
\begin{equation*}
\ds\frac{\partial u_i}{\partial t}=d\frac{\partial^2 u_i}{\partial x^2}-q\frac{\partial u_i}{\partial x}+u_i(r-u_i),\;\;i=1,2,3,
\end{equation*}
where $u_i$ is the density of the species in segment $i$. At the junction point $x=-L_3$, the continuity condition and the flux balancing condition are imposed, and we omit them for simplicity.
A geometric method was used in \cite{45-Vasilyeva-2019} to show the existence and uniqueness of  positive steady state. Moreover, the authors in \cite{42-Jin-Peng-Shi-2019} obtained the persistence condition for a single species in general river networks.
\begin{figure}[htbp]
	\centering\includegraphics[width=7cm]{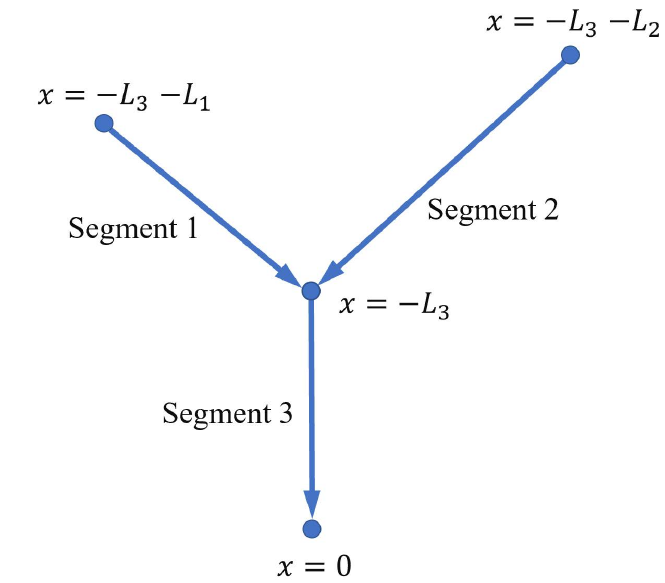}
	\caption{A Y-shaped river network \cite{45-Vasilyeva-2019}.} \label{fig3}
\end{figure}

To our best knowledge, for river networks with branches, there are few results on question (Q2) in the framework of RDA equations, and the results on question (Q2) for patch models mainly focus on the case $n=3$.
In fact, for appropriate matrices $(D_{kj})$ and $(Q_{kj})$, patch model \eqref{App1} can  be used to model
the interaction between two competing species in river networks with branches.
For example, if $n=3$, there are three types of the river networks, and for spatially heterogeneous environments (replace $r$ by $r_i$ in model \eqref{App1}), the authors in  \cite{32-Jiang-Lam-Lou-2020,46-Liu-Liu-Chen-2024} showed that: (i$_{ h}$) fixing $q_1=q_2=q$, there exists a critical value $\underline q$ such that the species with slower dispersal rate has competitive advantages for $q<\underline q$, and both network geometry and spatial heterogeneity have effect on the local and global dynamics for large drift rate; (ii$_{ h}$) fixing $d_1=d_2=d$, the species with slower drift rate has competitive advantages.

In this paper, we aim to study  the dynamics of model \eqref{App1} in a  Y-shaped river network with two branches, where the
number $n$ of patches is finite but arbitrary (see Figure \ref{fig1} for $n=8$).
For rivers without branches (see Figure \ref{fig2}), the authors in \cite{37-Chen-Liu-Wu-2022} showed that: (i) fixing $q_1=q_2=q$, the species with faster dispersal rate has competitive advantages, which is different from the above spatially heterogeneous case (see (i$_{ h}$)); (ii) fixing $d_1=d_2=d$, the species with slower drift rate has competitive advantages, which is similar to the above spatially heterogeneous case (see (ii$_{ h}$)). Our partial result in this paper
implies that (ii) also holds for a Y-shaped river network (Corollary \ref{4.7}).
It is still open whether (i) holds, and the authors in \cite{33-Jiang-Lam-Lou-2021} conjecture that (i) holds if the drift pattern $(Q_{kj})$ is not divergence free, i.e.,
there exists $1\le k\le n$ such that $\sum_{j\ne k}Q_{kj}\ne\sum_{j\ne k} Q_{jk}$. Our paper provides an initial step toward understand the dynamics of two-species competition model in river networks with branches.


Now we list some notations used throughout the paper. Denote $\mathbb N_0:=\{0,1,2,3,\cdots\}$. For $\gamma_1,\gamma_2\in \mathbb N_0$, the notation $\gamma_1\le k\le \gamma_2$ means that $$k\in M:=\{s\in \mathbb N_0:\gamma_1\le s\le \gamma_2\},$$ where
$M=\emptyset$ if $\gamma_1>\gamma_2$. For $\bm u=(u_1,\cdots,u_m)\in \mathbb{R}^m$, where $m$ is a positive integer, we write $\bm u\gg\bm0$ if $u_i>0$ for all $1\le i\le m$.

The rest of the paper is organized as follows. In Sect. 2, we rewrite model \eqref{App1} in another form and introduce some notions of a Y-shaped river
network for later use. In Sect. 3, we show that, under certain condition, there exist no positive equilibrium for model \eqref{App1} in the Y-shaped river
network. Then, we obtain the global dynamics in Sect. 4. Finally, we give some numerical simulations and conclusion remarks in Sect. 5.

\section{Model}\label{model2r}
In this section, we introduce some notations for a Y-shaped river
	network shown in Figure \ref{fig1} and rewrite model \eqref{App1} for later use. The notations are motivated by the RDA model in \cite{42-Jin-Peng-Shi-2019}.
\begin{figure}[htbp]
	\centering
	\includegraphics[width=9cm]{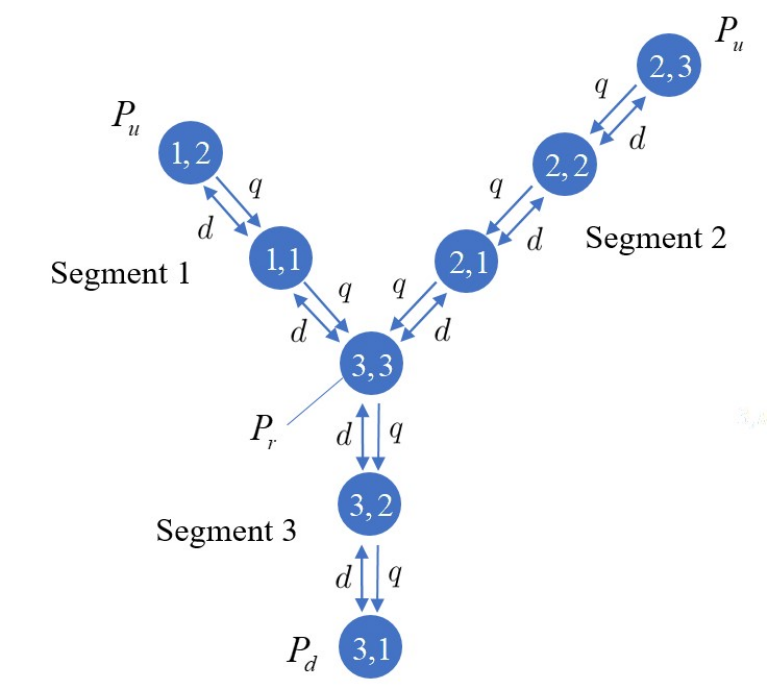}
	\caption{A Y-shaped river network. Here $m_1=2$, $m_2=3$, $m_3=3$, and $d$ and $q$ are the dispersal rate and the drift rate of one species, respectively.} \label{fig1}
\end{figure}
The Y-shaped river network in Figure \ref{fig1} consists of two branches (denoted by river segments 1 and 2) flowing into the main river (denoted by river segment 3). The set of all
patches is denoted by
\begin{equation}\label{pp}
P=\{(i,k):i=1,2,3,k\in N_i\}.
\end{equation}
Here
\begin{equation}
N_i=\{1,2,\cdots,m_i\}\quad\text{for}\quad i=1,2,3,
\end{equation}
where $1\le m_1\le m_2$ and $m_3\ge 2$. The first coordinate $i$ denotes that the patch $(i,k)$ is located in river segment $i$, while the second coordinate $k$ represents that it is the $k$-th patch in segment $i$.
The river network shown in Figure \ref{fig1}  can be viewed as the discrete form of that in Figure \ref{fig3}.

We divide the  patches in Figure \ref{fig1}  into four types:
\begin{equation}\label{type}
\begin{split}
&\text{(Upstream end patches)}\;\;\;\;P_u=\{(1,m_1),(2,m_2)\},\\
&\text{(Downstream end patch)} \;\;\;\;P_d=\{(3,1)\},\\
&\text{(Junction patch)}\;\;\;\;P_r=\{(3,m_3)\}, \\
&\text{(Interior patches)}\;\;\;\;P_o=P\setminus (P_d\cup P_u\cup P_r).
 \end{split}
\end{equation}
Denote
\begin{equation*}
\begin{split}
&\bm u(t)=\left(u^1_{1}(t),\cdots,u^1_{m_1}(t),u_{1}^2(t),\cdots,u^2_{m_2}(t),u^{3}_1(t),\cdots,u^3_{m_3}(t)\right),\\
&\bm v(t)=\left(v^1_{1}(t),\cdots,v^1_{m_1}(t),v_{1}^2(t),\cdots,v^2_{m_2}(t),v^{3}_1(t),\cdots,v^3_{m_3}(t)\right),
\end{split}
\end{equation*}
where $u^i_k(t)$ and $v^i_k(t)$ denote the numbers of two competing species $u$ and $v$ in patch $(i,k)$ at time $t$, respectively.
Then model \eqref{App1} can be rewritten as follows: (replace $w$ and $z$ by $u$ and $v$)
\begin{equation}\label{pat-r}
\begin{cases}
\ds\frac{{\rm d} u^i_k}{{\rm d}t}=d_1u^{i}_{k-1}-(2d_1+q_1)u^i_k+(d_1+q_1)u^{i}_{k+1}+u^i_k\left(r-u^i_k-v^i_k\right), &(i,k)\in P_o,\\
\ds\frac{{\rm d} v^i_k}{{\rm d}t}=d_2v^{i}_{k-1}-(2d_2+q_2)v^i_k+(d_2+q_2)v^{i}_{k+1}+v^i_k\left(r-u^i_k-v^i_k\right), &(i,k)\in P_o,\\
\ds\frac{{\rm d} u^i_k}{{\rm d}t}=-(d_1+q_1)u^i_k+d_1u^{i}_{k-1}+u^i_k\left(r-u^i_k-v^i_k\right), &(i,k)\in P_u,\\
\ds\frac{{\rm d} v^i_k}{{\rm d}t}=-(d_2+q_2)v^i_k+d_2v^{i}_{k-1}+v^i_k\left(r-u^i_k-v^i_k\right), &(i,k)\in P_u,\\
\ds\frac{{\rm d} u^i_k}{{\rm d}t}=-d_1u^i_k+(d_1+q_1)u^{i}_{k+1}+u^i_k\left(r-u^i_k-v^i_k\right), &(i,k)\in P_d,\\
\ds\frac{{\rm d} v^i_k}{{\rm d}t}=-d_2v^i_k+(d_2+q_2)v^{i}_{k+1}+v^i_k\left(r-u^i_k-v^i_k\right), &(i,k)\in P_d,\\
\ds\frac{{\rm d} u^i_k}{{\rm d}t}=d_1u^{i}_{k-1}-(3d_1+q_1)u^i_{k}+(d_1+q_1)(u^{1}_1+u^2_1)+u^i_k\left(r-u^i_k-v^i_k\right),&(i,k)\in P_r,\\
\ds\frac{{\rm d} v^i_k}{{\rm d}t}=d_2v^{i}_{k-1}-(3d_2+q_2)v^i_{k}+(d_2+q_2)(v^{1}_1+v^2_1)+v^i_k\left(r-u^i_k-v^i_k\right),&(i,k)\in P_r,\\
\bm u(0)=\bm u_0\ge(\not\equiv)\bm0,\;\bm v(0)=\bm v_0\ge(\not\equiv)\bm0,
\end{cases}
\end{equation}
where
$d_1$ and $d_2$ denote the dispersal rates of species $u$ and $v$, respectively; $q_1$ and $q_2$ denote the drift rates of species $u$ and $v$, respectively. The two species have the same intrinsic growth rate, denoted by $r>0$, which indicates that they have enough
resource.

\section{Nonexistence of positive equilibria}
In this section, we show the nonexistence of positive equilibria for model \eqref{pat-r}, a critical step in analyzing its global dynamics.
In Section \ref{sub31}, we present our main result as Theorem \ref{priori2}. Section \ref{sub22} provides some \textit{a priori} estimates on the positive equilibrium, which are crucial for the proof of Theorem \ref{priori2}. Finally, The proof of Theorem \ref{priori2} is detailed in Section \ref{sub33}.
\subsection{Main result}\label{sub31}
Define the set $\mathcal S$ as the union of two subsets $\mathcal S_1 $ and $\mathcal S_2$, as follows:
\begin{equation}\label{S}
	\mathcal S:=\mathcal S_1 \cup\mathcal S_2,
\end{equation}
where
\begin{equation}\label{S1S2}
	\begin{split}
		&\mathcal S_{1}:=\{(d,q):\; 0<d \le d_1,0<q\le \ds\f{q_1}{d_1}d,(d,q)\ne(d_1,q_1)\},\\
		&\mathcal S_2:=\{(d,q):\; d\ge d_1,q\ge  \frac{q_1}{d_1}d,(d,q)\ne(d_1,q_1)\}.\\
	\end{split}
\end{equation}
\begin{figure}[htbp]
	\centering\includegraphics[width=9cm]{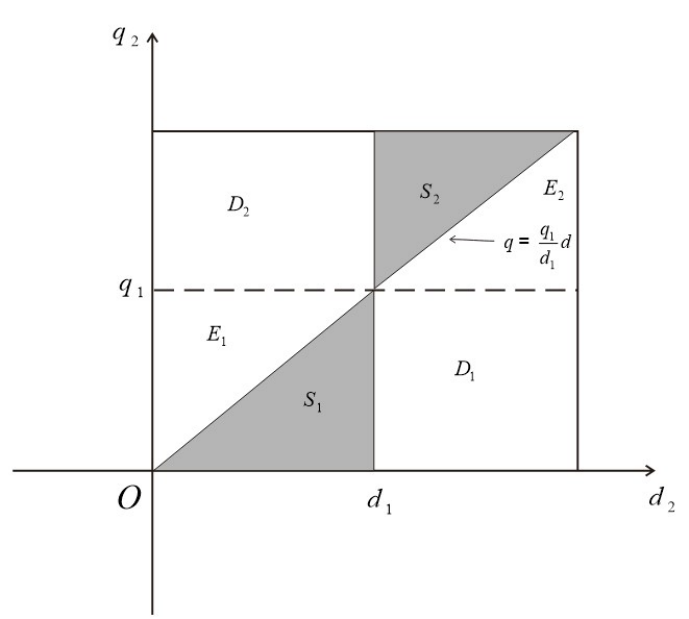}
	\caption{Illustration of $\mathcal S_{1}$ and $\mathcal S_{2}$.} \label{fig4}
\end{figure}
These subets are illustrated in Figure \ref{fig4}. Our main result for this section is presented below.
\begin{theorem}\label{priori2}
	Assume that $d_1,q_1>0$ and $(d_2,q_2)\in\mathcal S$. Then model
	\eqref{pat-r} admits no positive equilibria.
\end{theorem}
We will prove Theorem \ref{priori2} in Section \ref{sub33}.
\subsection{\textit{A priori} estimates}\label{sub22}
Suppose that model \eqref{pat-r} admits a positive equilibrium, denoted by $(\bm u,\bm v)$ with $\bm u,\bm v\gg\bm 0$,
where
\begin{equation}\label{uuuvv}
\begin{split}
&\bm u=(u^1_{1},\cdots,u^1_{m_1},u_{1}^2,\cdots,u^2_{m_2},u^{3}_1,\cdots,u^3_{m_3}),\\
&\bm v=(v^1_{1},\cdots,v^1_{m_1},v_{1}^2,\cdots,v^2_{m_2},v^{3}_1,\cdots,v^3_{m_3}).
\end{split}
\end{equation}
For simplicity of notations, we set
\begin{equation}\label{fff}
u^1_0=u^2_0:=u^3_{m_3}, \;\;v^1_0=v^2_0:=v^3_{m_3},
\end{equation}
and define two finite sequences $\left\{f^i_k\right\}_{(i,k)\in P^*}$ and $\left\{g^i_k\right\}_{(i,k)\in P^*}$ with
\begin{equation}\label{pstar}
P^*:=P\cup\{(i,k):(i,k-1)\in P_u\}=P\cup\{(1,m_1+1),(2,m_2+1)\}
\end{equation}
and
\begin{equation}\label{fg}
\begin{split}
&f^i_k=\begin{cases}
d_1u_{k-1}^i-(d_1+q_1)u_{k}^{i}, &(i,k)\in P_o\cup P_u\cup P_r,\\
0,&(i,k)\in P_d\cup\{(1,m_1+1),(2,m_2+1)\},\\
\end{cases}\\
&g^i_k=\begin{cases}
d_2v_{k-1}^i-(d_2+q_2)v_{k}^{i}, &(i,k)\in P_o\cup P_u\cup P_r,\\
0,&(i,k)\in P_d\cup\{(1,m_1+1),(2,m_2+1)\},\\
\end{cases}
\end{split}
\end{equation}
where $P$, $P_o$, $P_u$, $P_r$, and $P_d$ are defined in \eqref{pp} and \eqref{type}.
 By \eqref{fff} and \eqref{fg},
\begin{equation}
\begin{split}
&f_1^i=d_1u_{0}^i-(d_1+q_1)u^i_1=d_1u_{m_3}^3-(d_1+q_1)u^i_1,\;\;i=1,2,\\
&g_1^i=d_2v_{0}^i-(d_2+q_2)v^i_1=d_2v_{m_3}^3-(d_2+q_2)v^i_1,\;\;i=1,2.
\end{split}
\end{equation}
Then $(\bm u,\bm v)$ satisfies
\begin{subequations}\label{steady}
\begin{align}
&f^i_k-f^{i}_{k+1}=-u^i_k\left(r-u^i_k-v^i_k\right), \;\;\;\;\;\;\;\;(i,k)\in P_o\cup P_u\cup P_d,\label{steady-a}\\
&g^i_k-g^{i}_{k+1}=-v^i_k\left(r-u^i_k-v^i_k\right), \;\;\;\;\;\;\;\;(i,k)\in P_o\cup P_u\cup P_d,\label{steady-b}\\
&f^i_k-f_1^1-f^2_1=-u^i_k\left(r-u^i_k-v^i_k\right),\;\;\;(i,k)\in P_r,\label{steady-c}\\
&g^i_k-g_1^1-g^2_1=-v^i_k\left(r-u^i_k-v^i_k\right),\;\;\;\;(i,k)\in P_r.\label{steady-d}
\end{align}
\end{subequations}

It is worth noting that the two auxiliary sequences $\{f_k^i\}_{(i,k)\in P^*}$ and $\{g_k^i\}_{(i,k)\in P^*}$  are inspired by \cite{13-Zhou-2016}, where they were employed to demonstrate the nonexistence of positive steady states in PDE models for rivers without branches. Similar approaches can also be found in \cite{37-Chen-Liu-Wu-2022,46-Liu-Liu-Chen-2024} for patch models.
Specifically, by analyzing the signs (positive or negative) of these two sequences, one can derive contradictions, which implies the nonexistence of positive equilibria.

However, for a Y-shaped river network, each of these two auxiliary sequences comprises three subsequences, introducing additional technical challenges in analyzing their signs (positive or negative).
Below, we develop a method to estimate the signs of these auxiliary sequences for the Y-shaped river network illustrated in Figure \ref{fig1}.
The \textit{a priori} estimates of the two auxiliary sequences $\{f_k^i\}_{(i,k)\in P^*}$ and $\{g_k^i\}_{(i,k)\in P^*}$ are classified into two categories:
\begin{enumerate}
\item [(t$_1$)] estimates of $\{f_k^i\}_{(i,k)\in P^*}$ and $\{g_k^i\}_{(i,k)\in P^*}$ for a fixed branch (i.e., fixed $i$);
\item [(t$_2$)] estimates of $\{f_k^i\}_{(i,k)\in P^*}$ and $\{g_k^i\}_{(i,k)\in P^*}$  across the three branches.
\end{enumerate}
We first  present  two \textit{a priori} estimates of type (t$_1$).
The first estimate is derived using the method of upper and lower solutions.
\begin{lemma}\label{priori}
Assume that $d_1,q_1>0$, $(d_2,q_2)\in\mathcal S_1$, and let $\{f^i_k\}_{(i,k)\in P^\ast}$ and $\{g^i_k\}_{(i,k)\in P^\ast}$ be defined as in \eqref{fg}. Then the following two statements hold:
\begin{enumerate}
\item [\rm{(i)}] If there exists  $2\le l\le m_3$ such that $f^3_l\ge0$, then $g^3_l\ge0$ for $l=2$ and $g^3_l>0$ for $2<l\le m_3$;
\item [\rm{(ii)}]For each $i=1,2$, if there exists $1\le l\le m_i$ such that $f^i_l\le0$, then $g^i_l\le0$ for  $l=m_i$ and $g^i_l<0$ for $1\le l<m_i$.
\end{enumerate}
\end{lemma}
\begin{proof}
We only prove (i), and (ii) can be treated similarly. Note from \eqref{fg} that $f_{1}^3=g_1^3=0$.
If $l=2$, then we see from \eqref{steady-a}-\eqref{steady-b} that
${f_2^3}/{u_1^3}={g_2^3}/{v_1^3}$,
which implies $g^3_l\ge0$.

For $l>2$, suppose to the contrary that $g^3_l\le0$.
Now we consider the following auxiliary system:
\begin{equation}\label{auxi2-1}
\begin{cases}
\ds\frac{{\rm d} w_k}{{\rm d}t}=\ds\sum_{j=1}^{l-1}(d_1D_{kj}+q_1Q_{kj})w_j+w_k(r-w_k-z_k),&k=1,\cdots,l-1,\\
\ds\frac{{\rm d} z_k}{{\rm d}t}=\ds\sum_{j=1}^{l-1}(d_2D_{kj}+q_2Q_{kj})z_j+z_k(r-w_k-z_k),&k=1,\cdots,l-1,\\
\end{cases}
\end{equation}
where $(D_{kj})$ and $(Q_{kj})$ are $(l-1)\times (l-1)$ matrices with
\begin{equation}\label{3.2}
D_{kj}=\begin{cases}
1,	&k=j-1 \;\mbox{or}\; k=j+1,\\
-2, &k=j=2,\cdots,l-2,\\
-1, &k=j=1,l-1,\\
0,  &\mbox{otherwise},
\end{cases}
\qquad
Q_{kj}=\begin{cases}
	1,	&k=j-1,\\
	-1, &k=j=2,\cdots,l-1,\\
	0,  &\mbox{otherwise}.
\end{cases}
\end{equation}
The  equilibrium of \eqref{auxi2-1} satisfies
\begin{equation}\label{auxi2}
\begin{cases}
-\ds\sum_{j=1}^{l-1}(d_1D_{kj}+q_1Q_{kj})w_j-w_k(r-w_k-z_k)=0,&k=1,\cdots,l-1,\\
-\ds\sum_{j=1}^{l-1}(d_2D_{kj}+q_2Q_{kj})z_j-z_k(r-w_k-z_k)=0,&k=1,\cdots,l-1.\\
\end{cases}
\end{equation}
Define an order
\begin{equation}\label{order}
(\bm w^1, \bm z^1){\succeq} (\bm w^2, \bm z^2) \;\; \text{if}\;\;w^1_k\ge w^2_k\;\;\text{and}\;\;z^1_k\le z^2_k\;\;\text{for}\;\;k=1,\cdots,l-1.
\end{equation}
Then system \eqref{auxi2-1} generates a monotone dynamical system, which is order preserving.
Recall that $(\bm u,\bm v)$ is a positive equilibrium of \eqref{pat-r} with  $f^3_l\ge0$ and $g^3_l\le0$, where
\begin{equation*}
\begin{split}
&\bm u=(u^1_{1},\cdots,u^1_{m_1},u_{1}^2,\cdots,u^2_{m_2},u^{3}_1,\cdots,u^3_{m_3}),\\
&\bm v=(v^1_{1},\cdots,v^1_{m_1},v_{1}^2,\cdots,v^2_{m_2},v^{3}_1,\cdots,v^3_{m_3}).
\end{split}
\end{equation*}
It follows that
\begin{equation}\label{auxi2-lower}
\begin{cases}
-\ds\sum_{j=1}^{l-1}(d_1D_{kj}+q_1Q_{kj})u^3_j-u^3_k(r-u^3_k-v^3_k)=0,&k=1,\cdots,l-2,\\
-\ds\sum_{j=1}^{l-1}(d_1D_{kj}+q_1Q_{kj})u^3_j-u^3_k(r-u^3_k-v^3_k)=-f_{l}^3\le0,&k=l-1,\\
-\ds\sum_{j=1}^{l-1}(d_2D_{kj}+q_2Q_{kj})v^3_j-v^3_k(r-u^3_k-v^3_k)=0,&k=1,\cdots,l-2,\\
-\ds\sum_{j=1}^{l-1}(d_2D_{kj}+q_2Q_{kj})v^3_j-v^3_k(r-u^3_k-v^3_k)=-g_{l}^3\ge0,&k=l-1,
\end{cases}
\end{equation}
which implies that $(u^3_1,\cdots,u^3_{l-1},v^3_1,\cdots,v^3_{l-1})$ is a lower solution of \eqref{auxi2} (or sub-equilibrium of \eqref{auxi2-1}).

By Proposition \ref{lem-A1},
system \eqref{auxi2-1} has a semi-trivial equilibrium $(\bm{w}^\ast,\mathbf{0})$ with  $\bm{w}^\ast=(w^\ast_1,\cdots,w^\ast_{l-1})\gg\bm0$, which is unstable.
Denote by $\lambda_1$ the principal eigenvalue of the following eigenvalue problem:
\begin{equation}\label{3.1}
\sum_{j=1}^{l-1}(d_2D_{ij}+q_2Q_{ij})\phi_j+(r-w^\ast_i)\phi_i=\lambda\phi_i,\;i=1,\cdots,l-1,
\end{equation}
and denote the corresponding eigenvector by $\bm \psi^T$
with $\bm \psi=(\psi_1,\cdots,\psi_{l-1})\gg\bm0$. Since
$(\bm{w}^\ast,\mathbf{0})$ is unstable, it follows that $\lambda_1>0$.
We first choose $\varepsilon_1>0$ so that
\begin{equation}\label{compare1}
\la_1>\varepsilon_1\psi_k,\;v^3_k>\varepsilon_1\psi_k\;\;\text{for all} \;\;1\le k\le l-1.
\end{equation}
Then we show that, for $ 0<\varepsilon_2\ll1$,
$(\bm{w}^\ast+\varepsilon_2\bm{1},\varepsilon_1\bm\psi)$
is an upper solution of \eqref{auxi2} (or super-equilibrium of \eqref{auxi2-1}), where $\bm {1}=\left(1,\cdots,1 \right)$.
A direct computation yields, for $k=1,\cdots,l-1$,
\begin{equation*}
\begin{split} &-\ds\sum_{j=1}^{l-1}(d_1D_{kj}+q_1Q_{kj})(w_j^*+\varepsilon_2)-(w_k^*+\varepsilon_2)(r-w_k^*-\varepsilon_2-\varepsilon_1\psi_k)=\mathcal R_k^1(\varepsilon_1,\varepsilon_2),\\ &-\ds\sum_{j=1}^{l-1}(d_2D_{kj}+q_2Q_{kj})\varepsilon_1\psi_j-\varepsilon_1\psi_k(r-w_k^*-\varepsilon_2-\varepsilon_1\psi_k)= \varepsilon_1\psi_k\mathcal R_k^2(\varepsilon_1,\varepsilon_2),
\end{split}
\end{equation*}
where
\begin{equation*}
\begin{split}
&\mathcal R_k^1(\varepsilon_1,\varepsilon_2):=\varepsilon_2\left[-q_1\left(\sum_{j=1}^{l-1}Q_{kj}\right)-r+2w_k^*+\varepsilon_2+\varepsilon_1\psi_k\right]+\varepsilon_1\psi_kw_k^*,\\
&\mathcal R_k^2(\varepsilon_1,\varepsilon_2):=-\la_1+\varepsilon_2+\varepsilon_1\psi_k.
\end{split}
\end{equation*}
Since
\begin{equation*}
\lim_{\varepsilon_2\to0}\mathcal R_k^1(\varepsilon_1,\varepsilon_2)>0,\;\;\lim_{\varepsilon_2\to0}\mathcal R_k^2(\varepsilon_1,\varepsilon_2)<0,\;\;k=1,\cdots,l-1,
\end{equation*}
it follows that
$(\bm{w}^\ast+\varepsilon_2\bm{1},\varepsilon_1\bm\psi)$
is an upper solution of \eqref{auxi2} for $0<\varepsilon_2\ll1$.
 By \eqref{auxi2-lower}, we see that $(u^3_1,\cdots,u^3_{l-1})$ is a lower solution of
$$
 -\ds\sum_{j=1}^{l-1}(d_1D_{kj}+q_1Q_{kj})u^3_j-u^3_k(r-u^3_k)=0,\;\;k=1,\cdots,l-1,
 $$
 which implies $u^3_k\le w_k^*$, and consequently, $u^3_k<w^\ast_k+\varepsilon_2$ for $k=1,\cdots,l-1$. This combined with
 \eqref{compare1}
 implies that
 $$(\bm{w}^\ast+\varepsilon_2\bm{1},\varepsilon_1\bm\psi)\succeq (\not =)(u^3_1,\cdots,u^3_{l-1},v^3_1,\cdots,v^3_{l-1}),$$ where the order \lq\lq$\succeq$'' is defined in \eqref{order}.
 Then, by \cite[Lemma 1.1]{43-Hess-1992}, there exists a positive equilibrium for system \eqref{auxi2-1}, which contradicts
Proposition \ref{lem-A2} (see Appendix). This completes the proof.
\end{proof}

The second estimate of type (t$_1$) is based on the characteristics of equation \eqref{steady}.
\begin{lemma}\label{canot}
Assume that $d_1,q_1>0$ and $(d_2,q_2)\in \mathcal S_1$, and let $\{f^i_k\}_{(i,k)\in P^\ast}$ and $\{g^i_k\}_{(i,k)\in P^\ast}$ be defined as in \eqref{fg}. Then, for fixed $i=1,2$ (resp. $i=3$), the following case cannot occur:
there exist $l_*$ and $l^*$ with $1\le l_*< l^*\le m_i$ (resp. $1\le l_*< l^*\le m_i-1$) such that
\begin{subequations}\label{lem21}
\begin{align}
&f_k^i,\;g_k^i\ge0\;\;\text{for}\;\; l_*+1\le k\le l^*;\label{lem21-a}\\
&\min\{f^i_{k},\;g^i_{k}\}\le0\;\;\text{for}\;\;k=l_*,\;l^*+1.\label{lem21-b}
\end{align}
\end{subequations}
\end{lemma}
\begin{proof}
Suppose to the contrary that this case occurs. Since $(i,k)\in P_o\cup P_u\cup P_d$ for $l_*\le k\le l^*$, it follows from \eqref{steady} that, for $l_*\le k\le l^*$,
\begin{equation}\label{pro1}
\begin{split}
&f^i_k-f^{i}_{k+1}=-u^i_k\left(r-u^i_k-v^i_k\right),\\
&g^i_k-g^{i}_{k+1}=-v^i_k\left(r-u^i_k-v^i_k\right).\\
\end{split}
\end{equation}
By \eqref{lem21} and \eqref{pro1} with $k=l_*,l^*$,
\begin{equation}\label{last}
u^i_{l_*}+v^i_{l_*}\le r\;\;\text{and}\;\;u^i_{l^*}+v^i_{l^*}\geq r.
\end{equation}
In addition, we see from \eqref{lem21-a} that
$$u^i_{l_*}>\cdots>u^i_{l^*}\;\;\text{and}\;\;v^i_{l_*}>\cdots>v^i_{l^*},$$
which contradicts \eqref{last}. This completes the proof.
\end{proof}
To establish \textit{a priori} estimates of type (t$_2$), we first derive an identity for $\left\{f^i_k\right\}_{(i,k)\in P^*}$ and $\left\{g^i_k\right\}_{(i,k)\in P^*}$.
\begin{lemma}\label{impident}
For $1\le {l_3}\le m_3$ and $0\le {l_i}\le m_i$ $(i=1,2)$, the following identity holds:
\begin{equation}\label{fgiden3}
\begin{split}
&\left(f_{{l_3}}^3v_{l_3}^3-g_{{l_3}}^3u_{{l_3}}^3\right)\left(\f{d_1+q_1}{d_1}\right)^{{l_3}-m_3}-
\sum_{i=1}^2\left(f^i_{{l_i}+1}v^i_{{l_i}}-g^i_{{l_i}+1}u^i_{{l_i}}\right)\left(\f{d_1+q_1}{d_1}\right)^{{l_i}}\\
=&\frac{1}{d_1}\sum_{i=1}^2\sum_{k=1}^{{l_i}}h_k^if^i_k\left(\f{d_1+q_1}{d_1}\right)^{k-1}+\frac{1}{d_1}\sum_{k={l_3}+1}^{m_3}h_k^3f^3_k\left(\f{d_1+q_1}{d_1}\right)^{k-m_3-1},
\end{split}
\end{equation}
where $\left\{f^i_k\right\}_{(i,k)\in P^*}$ and $\left\{g^i_k\right\}_{(i,k)\in P^*}$ are defined in \eqref{fg}, and
\begin{equation}\label{3.3-2}
			h^i_k:=\begin{cases}
				(d_1-d_2)(v_{k-1}^i-v_{k}^{i})-(q_1-q_2)v_{k}^{i}, &(i,k)\in P_o\cup P_u\cup P_r,\\
				0,&(i,k)\in P_d\cup\{(1,m_1+1),(2,m_2+1)\}.\\
			\end{cases}
	\end{equation}
\end{lemma}
\begin{proof}
Let
	\begin{equation}\label{3.3}
			\bar{g}^i_k:=\begin{cases}
				d_1v_{k-1}^i-(d_1+q_1)v_{k}^{i}, &(i,k)\in P_o\cup P_u\cup P_r,\\
				0,&(i,k)\in P_d\cup\{(1,m_1+1),(2,m_2+1)\}.\\
			\end{cases}\\
\end{equation}
It follows from
\eqref{steady-a} and \eqref{steady-c}  that
\begin{subequations}\label{2.29}
	\begin{align}
		&f^3_k-f^{3}_{k+1}=-u^3_k\left(r-u^3_k-v^3_k\right), \;\;\; l_3\leq k\leq m_3-1,\label{2.29-a}\\
		&f^3_{m_3}-f_1^1-f^2_1=-u^3_{m_3}\left(r-u^3_{m_3}-v^3_{m_3}\right),\label{2.29-b}\\
        &f^1_k-f^{1}_{k+1}=-u^1_k\left(r-u^1_k-v^1_k\right), \;\;\; 1\leq k\leq l_1,\label{2.29-c}\\
        &f^2_k-f^{2}_{k+1}=-u^2_k\left(r-u^2_k-v^2_k\right), \;\;\; 1\leq k\leq l_2.\label{2.29-d}
	\end{align}
\end{subequations}
Multiplying \eqref{2.29-a} by $\left(\dfrac{d_1+q_1}{d_1} \right)^kv^3_k$ and summing them from $k=l_3$ to $k=m_3-1$, we have
	\begin{equation}\label{3.4}
		\begin{split}
			-\sum_{k=l_3}^{m_3-1}u^3_k v^3_k\left(r-u^3_k-v^3_k \right) \left(\dfrac{d_1+q_1}{d_1} \right)^k=&\sum_{k=l_3}^{m_3-1}\left(f^3_k-f^3_{k+1} \right) v^3_k
			\left(\dfrac{d_1+q_1}{d_1} \right)^k\\
			=&f^3_{l_3} v^3_{l_3} \left(\dfrac{d_1+q_1}{d_1} \right)^{l_3}
			-\sum_{k=l_3+1}^{m_3-1}f^3_k \bar{g}^3_k \frac{(d_1+q_1)^{k-1}}{d_1^k}\\
			&-f^3_{m_3} v^3_{m_3-1} \left(\dfrac{d_1+q_1}{d_1} \right)^{m_3-1}
			.
		\end{split}
	\end{equation}
	Multiplying \eqref{2.29-b} by $\left(\dfrac{d_1+q_1}{d_1} \right)^{m_3}v^3_{m_3}$, we obtain that
	\begin{equation}\label{3.5}
		-u^3_{m_3} v^3_{m_3}\left(r-u^3_{m_3}-v^3_{m_3}\right) \left(\dfrac{d_1+q_1}{d_1} \right)^{m_3}
		=\left(f^3_{m_3}-f^1_{1}-f^2_1\right) v^3_{m_3}
		\left(\dfrac{d_1+q_1}{d_1} \right)^{m_3}.
	\end{equation}
	Multiplying \eqref{2.29-c} by $\left(\dfrac{d_1+q_1}{d_1} \right)^{m_3+k}v^1_k$ and summing them from $k=1$ to $k=l_1$, we obtain that
	\begin{equation}\label{3.6}
		\begin{split}
	-&\sum_{k=1}^{l_1}u^1_k v^1_k\left(r-u^1_k-v^1_k \right) \left(\dfrac{d_1+q_1}{d_1} \right)^{m_3+k}\\ =&\sum_{k=1}^{l_1}\left(f^1_k-f^1_{k+1} \right) v^1_k
			\left(\dfrac{d_1+q_1}{d_1} \right)^{m_3+k}\\
			=&f^1_{1} v^1_{1} \left(\dfrac{d_1+q_1}{d_1} \right)^{m_3+1}
			-f^1_{l_1+1} v^1_{l_1} \left(\dfrac{d_1+q_1}{d_1} \right)^{m_3+l_1}\\
			&-\sum_{k=2}^{l_1}f^1_k \bar{g}^1_k \frac{(d_1+q_1)^{m_3+k-1}}{d_1^{m_3+k}}.
		\end{split}
	\end{equation}
Similarly, we have
	\begin{equation}\label{2.31}
	\begin{split}
-&\sum_{k=1}^{l_2}u^2_k v^2_k\left(r-u^2_k-v^2_k \right) \left(\dfrac{d_1+q_1}{d_1} \right)^{m_3+k}\\
=&f^2_{1} v^2_{1} \left(\dfrac{d_1+q_1}{d_1} \right)^{m_3+1}
		-f^2_{l_2+1} v^2_{l_2} \left(\dfrac{d_1+q_1}{d_1} \right)^{m_3+l_2}\\
		&-\sum_{k=2}^{l_2}f^2_k \bar{g}^2_k \frac{(d_1+q_1)^{m_3+k-1}}{d_1^{m_3+k}}.
	\end{split}
\end{equation}
	Summing \eqref{3.4}-\eqref{2.31} yields
	\begin{equation}\label{3.7}
		\begin{split}
			-&\sum_{k=l_3}^{m_3}u^3_k v^3_k\left(r-u^3_k-v^3_k \right) \left(\dfrac{d_1+q_1}{d_1} \right)^k
			-\sum_{i=1}^2\sum_{k=1}^{l_i}u^i_k v^i_k\left(r-u^i_k-v^i_k \right) \left(\dfrac{d_1+q_1}{d_1} \right)^{m_3+k}\\
			=&f^3_{l_3} v^3_{l_3} \left(\dfrac{d_1+q_1}{d_1} \right)^{l_3}
			-\sum_{k=l_3+1}^{m_3}f^3_k \bar{g}^3_k \frac{(d_1+q_1)^{k-1}}{d_1^k}\\
			&-\sum_{i=1}^2f^i_{l_i+1} v^i_{l_i} \left(\dfrac{d_1+q_1}{d_1} \right)^{m_3+l_i}
			-\sum_{i=1}^2\sum_{k=1}^{l_i}f^i_k \bar{g}^i_k \frac{(d_1+q_1)^{m_3+k-1}}{d_1^{m_3+k}}.
		\end{split}
	\end{equation}
Similar to \eqref{2.29}, we see from \eqref{steady-b} and \eqref{steady-d} that
\begin{equation}\label{2.32}
	\begin{split}
		&g^3_k-g^{3}_{k+1}=-v^3_k\left(r-u^3_k-v^3_k\right), \;\;\; l_3\leq k\leq m_3-1,\\
		&g^3_{m_3}-g_1^1-g^2_1=-v^3_{m_3}\left(r-u^3_{m_3}-v^3_{m_3}\right),\\
		&g^1_k-g^{1}_{k+1}=-v^1_k\left(r-u^1_k-v^1_k\right), \;\;\; 1\leq k\leq l_1,\\
		&g^2_k-g^{2}_{k+1}=-v^2_k\left(r-u^2_k-v^2_k\right), \;\;\; 1\leq k\leq l_2.
	\end{split}
\end{equation}
Then, using similar arguments as in the proof of \eqref{3.7},  we deduce from \eqref{2.32} that
	\begin{equation}\label{3.8}
		\begin{split}
			-&\sum_{k=l_3}^{m_3}u^3_k v^3_k\left(r-u^3_k-v^3_k \right) \left(\dfrac{d_1+q_1}{d_1} \right)^k-\sum_{i=1}^2\sum_{k=1}^{l_i}u^i_k v^i_k\left(r-u^i_k-v^i_k \right) \left(\dfrac{d_1+q_1}{d_1} \right)^{m_3+k}\\
			=&g^3_{l_3} u^3_{l_3} \left(\dfrac{d_1+q_1}{d_1} \right)^{l_3}
			-\sum_{k=l_3+1}^{m_3}g^3_k f^3_k \frac{(d_1+q_1)^{k-1}}{d_1^k}\\
			&-\sum_{i=1}^2g^i_{l_i+1} u^i_{l_i} \left(\dfrac{d_1+q_1}{d_1} \right)^{m_3+l_i}
			-\sum_{i=1}^2\sum_{k=1}^{l_i}g^i_k f^i_k \frac{(d_1+q_1)^{m_3+k-1}}{d_1^{m_3+k}}.\\
		\end{split}
	\end{equation}
	It is easy to check that $g^i_k=\bar{g}^i_k-h^i_k$ for $(i,k)\in P^\ast$. Then, we obtain (\ref{fgiden3}) by taking the difference of (\ref{3.7}) and (\ref{3.8}).
\end{proof}
\begin{remark}
For any $k_1,k_2\in \mathbb N_0$ and any sequence $\{s_k\}_{k=0}^\infty\subset \mathbb{R}$, we
adopt the convention
\begin{equation}\label{abuse}
\ds\sum_{k=k_1}^{k_2} s_k=0\;\;\text{when}\;\; k_1>k_2.
\end{equation}
Using this notation,
we note that \eqref{fgiden3} holds if $l_1=0$ or $l_2=0$ or $l_3=m_3$. 
\end{remark}

To derive \textit{a priori} estimates of type (t$_2$),
it is also necessary to analyze  $\{f_k^i\}_{(i,k)\in P^*}$ and $\{g_k^i\}_{(i,k)\in P^*}$  across two upstream branches (i.e., $i=1,2$).

\begin{lemma}\label{canot2}
Assume that $d_1,q_1>0$ and $(d_2,q_2)\in \mathcal S_1$ with $\mathcal S_1$ defined in \eqref{S1S2}, and let $\{f^i_k\}_{(i,k)\in P^\ast}$ and $\{g^i_k\}_{(i,k)\in P^\ast}$ be defined as in \eqref{fg}. Then the following four cases cannot occur:
\begin{enumerate}
\item [\rm{(i)}]There exist constants $l_1,l_2$ with $1\le l_2\le l_1\le m_1$  such that
\begin{subequations}\label{lem23}
\begin{align}
&f_k^1,\;g_k^1\ge 0\;\;\text{for}\;\; 1\le k\le l_1;\;\;f_k^2,\;g_k^2\le0\;\text{for}\; 1\le k\le l_2;\label{lem23-a}\\
&\min\{f_{l_1+1}^1,g_{l_1+1}^1\}\le0\;\;\text{and}\;\;\max\{f_{l_2+1}^2,g_{l_2+1}^2\}\ge0.\label{lem23-b}
\end{align}
\end{subequations}

\item [\rm{(ii)}]There exist constants $l_1,l_2$ with $1\le l_1\le l_2\le m_2$  such that
\begin{subequations}\label{lem28}
\begin{align}
	&f_k^1,\;g_k^1\le 0\;\;\text{for}\;\; 1\le k\le l_1;\;\;f_k^2,\;g_k^2\ge0\;\text{for}\; 1\le k\le l_2;\label{lem28-a}\\
	&\max\{f_{l_1+1}^1,g_{l_1+1}^1\}\ge0\;\;\text{and}\;\;\min\{f_{l_2+1}^2,g_{l_2+1}^2\}\le0.\label{lem28-b}
\end{align}	
\end{subequations}

\end{enumerate}
\end{lemma}
\begin{proof}
(i) Suppose to the contrary that (i) occurs. Note that \eqref{pro1} also holds for $k=l_i\;(i=1,2)$.
Then we see from \eqref{lem23} that
\begin{equation}\label{cc2}
u^1_{l_1}+v^1_{l_1}\ge r\ge u^2_{l_2}+v^2_{l_2},
\end{equation}
By \eqref{lem23-a} again,
\begin{equation*}
\begin{split}
&u^1_{l_2}\le \left(\frac{d_1}{d_1+q_1}\right)^{l_2}u_{m_3}^3\le u^2_{l_2},\\
&v^1_{l_2}\le\left(\frac{d_2}{d_2+q_2}\right)^{l_2} v_{m_3}^3\le v^2_{l_2},\\
&u^1_{l_1}< \cdots<u^1_{1}<u_{m_3}^3,\;\;v^1_{l_1}< \cdots<v^1_{1}<v_{m_3}^3.
\end{split}
\end{equation*}
This combined with $l_2\le l_1$ yields
\begin{equation}\label{21cos}
u^1_{l_1}+v^1_{l_1}\le u^1_{l_2}+v^1_{l_2}\le u^2_{l_2}+v^2_{l_2}.
\end{equation}
Note from Lemma \ref{priori} (ii) that at least one element of sequence
 $\{g^2_k\}_{k=1}^{l_2}$ is negative if
 $l_2<m_2$ or $l_2=m_2$ with $m_2>1$.
Thus, one of the inequalities of \eqref{21cos} is strict if
 $l_2<m_2$ or $l_2=m_2$ with $m_2>1$ or $l_2<l_1$,
which contradicts \eqref{cc2}.

Then, in view of $1\le m_1\le m_2$, we only need to consider the case  $l_1=l_2=m_2=m_1=1$.
It follows from \eqref{cc2} and \eqref{21cos} that
\begin{equation}\label{21sin1-2}
u^2_{1}+v^2_{1}=r=u^1_{1}+v^1_{1}.
\end{equation}
Plugging \eqref{21sin1-2} into \eqref{pro1} and noticing that $f_2^2=g_2^2=f_2^1=g_2^1=0$,
we have
$f^1_1=g^1_1=f_1^2=g_1^2=0$. This implies that $(u_{1}^3,\cdots,u_{m_3}^3,v_{1}^3,\cdots,v_{m_3}^3)$ is a positive equilibrium of \eqref{App1} with $n=m_3$, which contradicts Proposition \ref{lem-A2} (see Appendix).
Therefore, (i) cannot occur.

(ii) Suppose to the contrary that (ii) occurs.
We only need to obtain a contradiction for the case that $1=l_1=l_2=m_1<m_2$, and other cases can be treated using similar arguments as in the proof of (ii).
For this case, \eqref{21sin1-2} also holds, and
plugging it into \eqref{pro1} with $(i,k)=(2,1)$,
we have
$f^2_2=f^2_1$ and $g^2_2=g^2_1$.
This combined with \eqref{lem28-a} implies that $f^2_2=f^2_1\ge0$ and $g^2_2=g^2_1\ge0$, and consequently,
$u^2_1>u^2_2$ and $v^2_1>v^2_2$, which deduce (by \eqref{21sin1-2}) that
$u^2_2+v^2_2<r$.
Then, by \eqref{pro1} and induction,
\begin{equation*}
f^2_k>0,\;\;g^2_k>0\;\;\text{for}\;\;k=3,\cdots,m_2+1,
\end{equation*}
which contradicts $f^2_{m_2+1}=g^2_{m_2+1}=0$. Therefore, (ii) cannot occur.
\end{proof}

We now establish \textit{a priori} estimates of type (t$_2$) by applying Lemmas \ref{impident} and \ref{canot2}.
\begin{lemma}\label{complex}
Assume that $d_1,q_1>0$ and $(d_2,q_2)\in \mathcal S_1$ with $\mathcal S_1$ defined in \eqref{S1S2}, and let $\{f^i_k\}_{(i,k)\in P^\ast}$ and $\{g^i_k\}_{(i,k)\in P^\ast}$ be defined as in \eqref{fg}. Then the following three cases cannot occur:
\begin{enumerate}
\item [\rm{(i)}]There exist constants $l_1, l_2, l_3$ with $0\le l_i\le m_i\;(i=1,2)$ and $1\le l_3\le m_3$ such that
\begin{subequations}\label{lem25}
\begin{align}
&f_k^i,\;g_k^i\le0\;\;\text{for}\;\;i=1,2\;\;\text{and}\;\; 1\le k\le l_i;\label{lem25-a}\\
&f_k^3,\;g_k^3\le0\;\;\text{for}\;\; l_3+1\le k\le m_3;\label{lem25-b}\\
&f_{l_3}^3,\;g_{l_1+1}^1,\;g_{l_2+1}^2\le 0\;\;\text{and}\;\;g_{l_3}^3,\;f_{l_1+1}^1,\;f_{l_2+1}^2\ge0.\label{lem25-c}
\end{align}
\end{subequations}
\item [\rm{(ii)}] There exist constants $ \tilde{l}_1, l_1, l_2, l_3$ with $1\le \tilde{l}_1\le l_1\le m_1$, $1\le l_2\le m_2$ and $1\le l_3\le m_3$ such that
\begin{subequations}\label{lem26}
\begin{align}
&f_k^1,\;g_k^1\ge0\;\;\text{for}\;\;1\le k\le \tilde{l}_1;\;\;f_k^1,\;g_k^1\le0\;\;\text{for}\;\;  \tilde{l}_1<k\le l_1;\label{lem26-a}\\
&f_k^2,\;g_k^2\le0\;\;\text{for}\;\; 1\le k\le l_2;\label{lem26-b}\\
&f_k^3,\;g_k^3\le0\;\;\text{for}\;\; l_3+1\le k\le m_3;\label{lem26-c}\\
&f_{l_3}^3,\;g_{l_1+1}^1,\;g_{l_2+1}^2\le 0\;\;\text{and}\;\;g_{l_3}^3,\;f_{l_1+1}^1,\;f_{l_2+1}^2\ge0.\label{lem26-d}
\end{align}
\end{subequations}
\item [\rm{(iii)}] There exist constants $l_1, \tilde{l}_2, l_2, l_3$ with $1\le \le m_1$, $1\le  \tilde{l}_2\le l_2\le m_2$ and $1\le l_3\le m_3$ such that
\begin{subequations}\label{lem27}
\begin{align}
&f_k^1,\;g_k^1\le0\;\;\text{for}\;\; 1\le k\le l_1;\label{lem27-a}\\
&f_k^2,\;g_k^2\ge0\;\;\text{for}\;\;1\le k\le \tilde{l}_2;\;\;f_k^2,\;g_k^2\le0\;\;\text{for}\;\;  \tilde{l}_2<k\le l_2;\label{lem27-b}\\
&f_k^3,\;g_k^3\le0\;\;\text{for}\;\; l_3+1\le k\le m_3;\label{lem27-c}\\
&f_{l_3}^3,\;g_{l_1+1}^1,\;g_{l_2+1}^2\le 0\;\;\text{and}\;\;g_{l_3}^3,\;f_{l_1+1}^1,\;f_{l_2+1}^2\ge0.\label{lem27-d}
\end{align}
\end{subequations}
\end{enumerate}
\end{lemma}
\begin{proof}
(i) Suppose to the contrary that (i) holds. By Lemma \ref{impident},
\begin{equation}\label{mathcalR1}
	\begin{split}
		\mathcal R:=&\frac{1}{d_1}\sum_{i=1}^2\sum_{k=1}^{l_i}h_k^if^i_k\left(\f{d_1+q_1}{d_1}\right)^{k-1}+\frac{1}{d_1}\sum_{k=l_3+1}^{m_3}h_k^if^3_k\left(\f{d_1+q_1}{d_1}\right)^{k-m_3-1}\\
=&\left(f_{l_3}^3v_{l_3}^3-g_{l_3}^3u_{l_3}^3\right)\left(\f{d_1+q_1}{d_1}\right)^{l_3-m_3}-
\sum_{i=1}^2\left(f^i_{l_i+1}v^i_{l_i}-g^i_{l_i+1}u^i_{l_i}\right)\left(\f{d_1+q_1}{d_1}\right)^{l_i}\\\leq&0,
	\end{split}
\end{equation}
where we have used \eqref{lem25-c} in the last step, and $\left\{h^i_k\right\}_{(i,k)\in P^*}$ is defined in \eqref{3.3-2}.

If $(d_2,q_2)\in\mathcal S_1$ with $d_2=d_1$,
then
$$h_k^i<0\;\;\text{for}\;\;1\leq k\leq l_i\;\;\text{with}\;\; i=1,2\;\;\text{and}\;\;l_3+1\leq k\leq m_3\;\;\text{with}\;\; i=3.$$ This combined with
\eqref{lem25-a}-\eqref{lem25-b} yields $\mathcal R\geq0$, and consequently, $\mathcal R=0$.

If $(d_2,q_2)\in\mathcal S_1$ with $d_2\ne d_1$, then $d_2<d_1$ and by \cite[Lemma 2.4]{13-Zhou-2016},
\begin{equation}\label{whos}
	\ds\f{q_1-q_2}{d_1-d_2}\ge\ds\f{q_2}{d_2},
\end{equation}
which yields
\begin{equation}\label{le2.3-s1}
h^i_k=(d_1-d_2)(v^i_{k-1}-v^i_k)-(q_1-q_2)v^i_k\le \frac{d_1-d_2}{d_2}g_k^i.
\end{equation}
This combined with \eqref{lem25-a}-\eqref{lem25-b} implies that $\mathcal R\geq0$, and consequently, $\mathcal R=0$.

For each of the above cases, we have
\begin{subequations}\label{le2.3-s2}
\begin{align}
&f_{l_3}^3=g_{l_3}^3=f_{l_1+1}^1=g_{l_1+1}^1=f_{l_2+1}^2=g_{l_2+1}^2=0,\\
&f^i_kg^i_k=0\;\;\text{for}\;\; 1\leq k\leq l_i\;\;\text{with}\;\; i=1,2\;\;\text{and}\;\; l_3+1\leq k\leq m_3\;\;\text{with}\;\;i=3.
\end{align}
\end{subequations}
This combined with \eqref{steady} implies that
$f_{k}^i=g_{k}^i=0$ for  $1\leq k\leq l_i+1$ with $i=1,2$ and $l_3\leq k\leq m_3$ with $i=3$. Therefore, $f_1^1=g_1^1=f_1^2=g_1^2=0$, and consequently, $(u_{1}^3,\cdots,u_{m_3}^3,v_{1}^3,\cdots,v_{m_3}^3)$ is a positive equilibrium of \eqref{App1} with $n=m_3$, which contradicts Proposition \ref{lem-A2} (see Appendix). Therefore, (i) cannot occur.

(ii) Now we prove (ii), while (iii) can be studied in a similar manner.
Suppose to the contrary that (ii) occurs.
If $\tilde{l}_1 \ge l_2$,
then
\eqref{lem23} holds with $l_1=\tilde{l}_1$, which contradicts Lemma \ref{canot2} (i).
Then we consider the case $\tilde{l}_1<l_2$.

\emph{Claim 1:} $f_k^1+f_k^2< 0$ for each $k=1,\cdots,\tilde{l}_1$.

\noindent\emph{Proof of Claim:}
Since $(1,\tilde{l}_1)\in P_o\cup P_u$, it follows from \eqref{steady-a} and \eqref{lem26-a} that
\begin{equation*}
u^1_{\tilde{l}_1}+v^1_{\tilde{l}_1}\geq r,
\end{equation*}
and
\begin{equation*}
u^3_{m_3}>u^1_1>\cdots>u^1_{\tilde{l}_1}\;\; \;\text{and}\;\;\; v^3_{m_3}>v^1_1>\cdots>v^1_{\tilde{l}_1},
\end{equation*}
which yields
\begin{equation}\label{3.9}
	u^3_{m_3}+v^3_{m_3}>r \;\;\; \text{and}\;\;\;u^1_{k}+v^1_{k}>r\;\;\;
	\text{for}\;\;\;k=1,\cdots,\tilde{l}_1-1.
\end{equation}
Since $\tilde{l}_1<l_2\le m_2$, we see from Lemma \ref{priori} (ii) that
\begin{equation}\label{g2k}
g^2_k<0 \;\;\text{for}\;\; k=1,\cdots,\tilde{l}_1.
\end{equation}
In addition, by (\ref{lem26-a}) and (\ref{lem26-b}), we have
\begin{equation}\label{3.11}
	\begin{split}
		&u^1_{k}\leq \left(\frac{d_1}{d_1+q_1}\right)^{k}u_{m_3}^3\le u^2_{k}\;\;\text{for}\;\; k=1,\cdots,\tilde{l}_1, \\
		&v^1_{k}\leq \left(\frac{d_1}{d_1+q_1}\right)^{k}v_{m_3}^3<v^2_{k}\;
		\;\text{for}\;\; k=1,\cdots,\tilde{l}_1,
	\end{split}
\end{equation}
which yields
\begin{equation}\label{2.33}
u^1_{k}+v^1_{k}< u^2_{k}+v^2_{k}\;\;\text{for}\;\; k=1,\cdots,\tilde{l}_1.
\end{equation}

Suppose to the contrary that $f_1^1+f_1^2\geq0$. Then, by \eqref{steady-c} and \eqref{3.9}, we have $f^3_{m_3}>0$, which contradicts (\ref{lem26-c}). Thus, $f_1^1+f_1^2<0$. By induction, it suffices to show that if $f_k^1+f_k^2<0$ for $k=1,\cdots,k_0$ with $1\leq k_0\le \tilde{l}_1$, then $f_{k_0+1}^1+f_{k_0+1}^2<0$.
By \eqref{steady-a},
\begin{equation*}
	(f_{k_0}^1+f_{k_0}^2)-(f_{k_0+1}^1+f_{k_0+1}^2)=
	-u^1_{k_0}(r-u^1_{k_0}-v^1_{k_0})-u^2_{k_0}(r-u^2_{k_0}-v^2_{k_0}).
\end{equation*}
Suppose to the contrary that $f_{k_0+1}^1+f_{k_0+1}^2\geq0$. Noticing that $f_{k_0}^1+f_{k_0}^2<0$, we see that at least
one of the two inequalities $u^1_{k_0}+v^1_{k_0}<r$ and $u^2_{k_0}+v^2_{k_0}<r$ holds. This combined with \eqref{2.33} yields $u^1_{k_0}+v^1_{k_0}<r$,
which contradicts \eqref{3.9}. This proves the claim.

\emph{Claim 2:} Define
$$\mathcal F_k:=h^1_kf^1_k+h^2_kf^2_k\;\;\text{for}\;\;k=1,\cdots,\tilde{l}_1,$$
where $\{h^i_k\}$ is defined in \eqref{3.3-2}. Then $\mathcal F_k>0$ for each $k=1,\cdots,\tilde{l}_1$.

\noindent\emph{Proof of Claim:}
We first show that
\begin{equation}\label{hkin}
h^1_k> h^2_k \;\;\text{for}\;\;k=1,\cdots,\tilde{l}_1,
\end{equation}
where $\{h_k^i\}$ is defined in \eqref{3.3-2}.
If $d_1=d_2$, we see from \eqref{3.11} that \eqref{hkin} holds.
Now we consider the case $d_1\ne d_2$.
By \eqref{3.11} again,
\begin{equation*}
	h^1_1
	=(d_1-d_2)(v^3_{m_3}-v^1_1)-(q_1-q_2)v^1_1
	>(d_1-d_2)(v^3_{m_3}-v^2_1)-(q_1-q_2)v^2_1
	= h^2_1.
\end{equation*}
For $k=2,\cdots,\tilde{l}_1$, noticing that $g^1_{k}\ge0,\; g^2_{k}<0$, we have
\begin{equation*}
	v^1_{k-1}-v^2_{k-1}>\left(1+\dfrac{q_2}{d_2} \right)(v^1_{k}-v^2_{k})\geq\left(1+\dfrac{q_1-q_2}{d_1-d_2} \right)(v^1_{k}-v^2_{k}) ,
\end{equation*}
where we have used \eqref{whos} and \eqref{3.11} in the last step.
This implies that
$h^1_k> h_k^2$ for $k=2,\cdots, \tilde{l}_1$. Thus, \eqref{hkin} holds.

If $(d_2,q_2)\in\mathcal S_1$ with $d_2=d_1$,  we have $h^2_k<0$ for $k=1,\cdots,  \tilde{l}_1$. If $(d_2,q_2)\in\mathcal S_1$ with $d_2\ne d_1$, we see from \eqref{le2.3-s1} that
\begin{equation}\label{h2s}
h^2_k\leq \frac{d_1-d_2}{d_2}g^2_k<0\;\;\text{for}\;\;k=1,\cdots, \tilde{l}_1,
\end{equation}
where we have used \eqref{g2k} in the last step.
Then it follows from \eqref{lem26-a} and \eqref{hkin} that
\begin{equation}\label{f1s}
	\mathcal F_k=h^1_kf^1_k+h^2_kf^2_k\ge h^2_k(f^1_k+f^2_k)>0\;\;\text{for}\;\;k=1,\cdots, \tilde{l}_1,
\end{equation}
where we have used Claim 1 and \eqref{h2s} in the second step. This proves the claim.

Using similar arguments as in the proof for the sign of \eqref{h2s}, we have
\begin{equation}\label{h13s}
h_k^i\begin{cases}
<0,&d_1=d_2\\
\le \frac{d_1-d_2}{d_2}g^i_k\leq0,&d_1\ne d_2
\end{cases}
\end{equation}
for $k=\tilde{l}_1+1,\cdots,l_i$ with $i=1,2$ and $k=l_3+1\cdots,m_3$ with $i=3$.
Then it follows from Lemma \ref{impident} that
\begin{equation}\label{3.12}
	\begin{split}
		0\ge &\left(f_{l_3}^3v_{l_3}^3-g_{l_3}^3u_{l_3}^3\right)\left(\f{d_1+q_1}{d_1}\right)^{l_3-m_3}-
		\sum_{i=1}^2\left(f^i_{l_i+1}v^i_{l_i}-g^i_{l_i+1}u^i_{l_i}\right)\left(\f{d_1+q_1}{d_1}\right)^{l_i}\\
		=&\frac{1}{d_1}\sum_{k=1}^{\tilde{l}_1}\mathcal{F}_{k}\left(\f{d_1+q_1}{d_1}\right)^{k-1}+\frac{1}{d_1}\sum_{i=1}^2\sum_{k=\tilde{l}_1+1}^{l_i}h^i_kf^i_k\left(\f{d_1+q_1}{d_1}\right)^{k-1}\\
		&+\frac{1}{d_1}\sum_{k=l_3+1}^{m_3}h^3_kf^3_k\left(\f{d_1+q_1}{d_1}\right)^{k-m_3-1}>0,
		\end{split}
\end{equation}
where we have used \eqref{lem26-d} in the first step and \eqref{lem26-a}-\eqref{lem26-c}, \eqref{h13s} and Claim 2 in the last step. This leads to a contradiction.
Therefore, (ii) cannot occur.
\end{proof}

\subsection{Proof of Theorem \ref{priori2}}\label{sub33}
This section aims to prove Theorem \ref{priori2}, namely, the nonexistence of positive equilibria in model \eqref{pat-r} for
$(d_2,q_2)\in\mathcal S$.
It suffices to prove the the nonexistence of positive equilibria for
$(d_2,q_2)\in\mathcal S_1$. Indeed, if $(d_2,q_2)\in \mathcal S_2$, then
$(d_1,q_1)\in\tilde{\mathcal S}_1$, where
\begin{equation*}
\tilde{\mathcal S}_{1}:=\left\{(d,q):\; 0<d \le d_2,0<q\le \ds\f{q_2}{d_2}d,(d,q)\ne(d_2,q_2)\right\}.
\end{equation*}
Note that the nonlinear terms of model \eqref{pat-r} are symmetric. If the nonexistence of positive equilibria for
$(d_2,q_2)\in\mathcal S_1$ is proven, the nonexistence of positive equilibria for $(d_2,q_2)\in\mathcal S_2$ can be derived
by interchanging the equations satisfied by $u_k^i$ and $v_k^i$ in model \eqref{pat-r}. 

We now focus on the case where $(d_2,q_2)\in\mathcal S_1$.
Suppose to the contrary that model \eqref{pat-r} admits a positive equilibrium $(\bm u,\bm v)$. We will derive a contradiction for each of the following three cases:
\begin{equation*}
\begin{split}
&{\rm(i)}\;\;f^3_{m_3}\le0,\;f^2_1\le0;\;\;{\rm(ii)}\;\;f^3_{m_3}\le0,\;f^2_1>0;\;\;{\rm(iii)}\;\;f^3_{m_3}>0.
\end{split}
\end{equation*}
To proceed, we first introduce some preliminary claims that will be utilized later.

\emph{Claim 1:} If $f^3_{m_3}\le0$, then $f^3_k\le0$ for $k=1,\cdots,m_3-1$.

\noindent\emph{Proof of Claim:} Note from \eqref{fg} that $f^3_1=0$. If the claim is not true, then
there exists two constants $k_*,k^*$ with $1\le k_*<k^*<m_3$ such that
$$f_k^3>0 \;\;\text{for}\;\;k=k_*+1,\cdots,k^*.$$
This combined with Lemma \ref{priori} (i) implies that $g_k^3\ge0$ for $k=k_*+1,\cdots,k^*$.
Then \eqref{lem21} holds with $i=3$, $l_*=k_*$ and $l^*=k^*$, which contradicts Lemma \ref{canot}. This proves the claim.

\emph{Claim 2:} If $f^3_{m_3}>0$, then $g^3_{m_3}>0$ and
\begin{equation}\label{2.36}
u^3_{m_3}+v^3_{m_3}<r.
\end{equation}

\noindent\emph{Proof of Claim:}
By Lemma \ref{priori} (i), we have $g^3_{m_3}>0$ if $m_3>2$. If $m_3=2$, noticing that $f^3_1=g^3_1=0$, we see from \eqref{steady-a} and \eqref{steady-b} that
${f_2^3}/{u^3_1}={g_2^3}/{v^3_1}$, which also implies that $g^3_{m_3}>0$.
Since $f_1^3=g_1^3=0$, it follows that
\begin{equation*}
	k_3:=\max\{1\le k\le m_3:\min\{f^3_k,g^3_k\}\leq 0\}
\end{equation*}
is well-defined with $1\leq k_3<m_3$, and  consequently, $(3,k_3)\in P_o\cup P_d$. Then, by  \eqref{steady-a} and \eqref{steady-b} again, we have $u^3_{k_3}+v^3_{k_3}\leq r$. Furthermore, by the definition of $k_3$,
\begin{equation*}
	u^3_{k_3}>u^3_{k_3+1}>\cdots>u^3_{m_3}\;\;\text{and}\;\;
	v^3_{k_3}>v^3_{k_3+1}>\cdots>v^3_{m_3},
\end{equation*}
which yields $u^3_{m_3}+v^3_{m_3}<r$. Therefore, the claim is true.

\emph{Claim 3:} For each $i=1,2$, if $f^i_1\le0$, then $g_{k}^i\le0$ for $k=1,\cdots,j_i+1$,
where
\begin{equation}\label{upper2}
j_i:=\begin{cases}
\min I_i-1, &\text{if}\;\;I_i\not=\emptyset\\
m_i,&\text{if}\;\;I_i=\emptyset
\end{cases}\;\;\text{with}\;\;I_i:=\{1\le k\le m_i+1:f^i_k>0\},
\end{equation}
and $1\le j_i\le m_i$.

\noindent\emph{Proof of Claim:}
If $I_i=\emptyset$, then $j_i=m_i$ and $f_k^i\le0$ for $k=1,\cdots,m_i+1$. This combined with
 Lemma \ref{priori} (ii) implies that the claim holds with $j_i=m_i$.
If $I_i\not=\emptyset$, by Lemma \ref{priori} (ii) again, we have
$g_{k}^i<0$ for $k=1,\cdots,j_i$, and it suffices to show that
 $g_{j_i+1}^i\le0$. Suppose to the contrary that $g_{j_i+1}^i>0$. Recall that $f^i_{m_i+1}=g^i_{m_i+1}=0$.
 Then $$ k_i:=\min\left\{k>j_i+1:\min\{f_k^i,g_k^i\}\le0\right\}$$ is well-defined with $j_i+1<k_i\le m_i+1$.
Then \eqref{lem21} holds with $l_*=j_i$ and $l^*=k_i-1$, which contradicts Lemma \ref{canot}. Therefore, the claim holds.

\emph{Claim 4:} For each $i=1,2$, if $f^i_1,g^i_1>0$, then the sequence $\{f^i_k\}_{k=1}^{j_i^{\ast}}$ changes signs at most once, where
\begin{equation}\label{3.13}
j_i^{\ast}:=\min\left\{1\le k\le m_i+1:f^i_k\ge0,g^i_k\le0\right\}-1,
\end{equation}
and $1\le j_i^{\ast}\le m_i$.

\noindent\emph{Proof of Claim:}
Since $f^i_{m_i+1}=g^i_{m_i+1}=0$, it follows that $j_i^{\ast}$ is well-defined.
Clearly, this claim holds if $j_i^{\ast}<3$. Next, we consider the case where $j_i^{\ast}\ge3$. If the claim does not hold in this case, then
\begin{equation*}
	s_1:=\min\{1\le k\le j_i^{\ast}:f_k^i<0\}\;\;\text{and}\;\;s_2:=\min\{s_1<k\le j_i^{\ast}:f_k^i>0\}
\end{equation*}
are well-defined with $1<s_1<s_2\le j_i^{\ast}$.
Define
\begin{equation*}
	s_3:=\begin{cases}
\min I_i^{\ast}-1, &\text{if}\;\;I_i^{\ast}\not=\emptyset\\
j_i^{\ast},&\text{if}\;\;I_i^{\ast}=\emptyset
\end{cases}\;\;\text{with}\;\;I_i^{\ast}:=\{s_2< k\le j_i^{\ast}:f_k^i<0\}.
\end{equation*}
By the definitions of $s_2$ and $s_3$,
we observe that
\eqref{lem21} holds with $l_*=s_2-1$ and $l^*=s_3$. This contradicts Lemma \ref{canot}, thereby proving the claim.

\emph{Claim 5:} For each $i=1,2$, if  $f_1^i,g_1^i>0$, then $ u^i_1+ v^i_1\ge r$.

\noindent\emph{Proof of Claim:}
Suppose to the contrary that
$r>u^i_1+v^i_1$. Note that $(i,k)\in P_o\cup P_u$ for $k=1,\cdots,m_i$ and $i=1,2$. Then, by
	\eqref{steady-a}-\eqref{steady-b} and induction,
	\begin{equation}\label{scont1}
		 f^i_{k},\;g^i_{k}>0\;\;\text{for}\;\;k=2,\cdots,m_i+1,
	\end{equation}
which contradicts $f^i_{m_i+1}=g^i_{m_i+1}=0$ (see \eqref{fg}).
This proves the claim.

Now we derive a contradiction for each of the cases (i)-(iii).

(i) For this case, $f^3_{m_3}\le0$ and $f^2_1\le0$.
It follows from Claim 1 that
\begin{equation}\label{i1-s1}
f_k^3\le0\;\;\text{for}\;\;1\le k\le m_3.
\end{equation}
Define
\begin{equation}\label{lower}
j_3=\begin{cases}
\max I_3, &\text{if}\;\;I_3\not=\emptyset\\
1,&\text{if}\;\;I_3=\emptyset
\end{cases}\;\;\text{with}\;\;I_3:=\{1\le k\le m_3:g^3_k>0\}.
\end{equation}
Then the following discussion is divided into three cases:
\begin{equation*}
{\rm (A_1)}\;\; f_1^1\le0;\;\; {\rm (A_2)}\;\; f_1^1>0,\;g_1^1\le0;\;\;{\rm (A_3)}\;\; f_1^1>0,\;g_1^1>0.
\end{equation*}
For case (A$_1$), it follows from Claim 3 that $j_i$ is well-defined in \eqref{upper2} with each $1\le j_i\le m_i$ for $i=1,2$. Furthermore,
\eqref{lem25} holds with $l_3=j_3$ and $l_i=j_i$ $(i=1,2)$, which contradicts Lemma \ref{complex} (i).

For case (A$_2$), by Claim 3 again, $j_2$ is well-defined in \eqref{upper2} with $1\le j_2\le m_2$, and \eqref{lem25} holds with $l_3=j_3,l_1=0,l_2=j_2$, which contradicts Lemma \ref{complex} (i) again.

For case (A$_3$), it follows from Claims 3-4 that $j_2$ is well-defined in \eqref{upper2} with $1\le j_2\le m_2$, $j_1^\ast$ is well-defined in \eqref{3.13} with $1\leq j_1^\ast\leq m_1$, and $\{f_k^1\}_{k=1}^{j_1^\ast}$ changes signs at most once.

If $\{f_k^1\}_{k=1}^{j_1^\ast}$ does not change sign, then
\begin{equation*}
	f^1_k,\;g^1_k\ge0 \;\;\text{for}\;\; 1\le k\le j_1^\ast,
\end{equation*}
and \eqref{lem26} holds with $\tilde{l}_1=l_1=j_1^\ast$, $l_2=j_2$, and $l_3=j_3$, which contradicts  Lemma \ref{complex} (ii).

If $\{f_k^1\}_{k=1}^{j_1^\ast}$ changes sign, then
$$\hat{j}_1:=\max\{1\le k\le j_1^\ast:f^1_k\ge0\}$$ is well-defined with $1\le \hat{j}_1< j_1^\ast$. Then, by Lemma \ref{priori} (ii) and the definition of $j_1^\ast$,
\begin{equation*}
	f_k^1,\;g^1_k\ge0\;\; \text{for}\;\; 1\leq k\leq \hat{j}_1 \;\; \text{and}\;\;
	f_k^1,\;g^1_k\leq0\;\; \text{for}\;\; \hat{j}_1< k\leq j_1^\ast,
\end{equation*}
and \eqref{lem26} holds with $\tilde{l}_1=\hat{j}_1$, $l_1=j_1^\ast$, $l_2=j_2$, and $l_3=j_3$, which also contradicts  Lemma \ref{complex} (ii).

(ii) For this case, $f^3_{m_3}\leq0$ and $f^2_1>0$. Clearly, \eqref{i1-s1} also holds, and
let  $j_3$ be defined in \eqref{lower}.
Then the following discussion is divided into three cases:
\begin{equation*}
	{\rm (B_1)}\;\; f_1^1>0,\;g^1_1>0;\;\; {\rm (B_2)}\;\;f_1^1>0,\;g_1^1\le0;\;\;{\rm (B_3)}\;\; f_1^1\le0.
\end{equation*}
For case (B$_1$), noticing that $f_1^1,f_1^2>0$ and $f^3_{m_3}\leq0$, it follows from \eqref{steady-c} that $u^3_{m_3}+v^3_{m_3}<r$.
Additionally, since $f^1_1,g^1_1>0$, we observe that $u^3_{m_3}>u^1_1$ and $v^3_{m_3}>v^1_1$, which yields $u^1_{1}+v^1_{1}<r$. This contradicts Claim 5.

For case (B$_2$), if $g_1^2\le0$, then \eqref{lem25} holds with $l_3=j_3$, $l_1=0$ and $l_2=0$, which contradicts Lemma \ref{complex} (i). Next, we consider the case $g_1^2>0$. Using arguments similar to those in the proof of case (B$_1$), we see that
 $u^3_{m_3}+v^3_{m_3}<r$. Furthermore, since $f^2_1,g^2_1>0$, we have $u^3_{m_3}>u^2_1$ and $v^3_{m_3}>v^2_1$, which implies that $u^2_{1}+v^2_{1}<r$. This also contradicts Claim 5.

For case (B$_3$), by Claim 3 again, $j_1$ is well-defined in \eqref{upper2} and $1\le j_1\le m_1$.
If $g_1^2\le0$, then
\eqref{lem25} holds with $l_3=j_3$, $l_1=j_1$ and $l_2=0$, which contradicts Lemma \ref{complex} (i).
If $g_1^2>0$, it follows from $f_1^2>0$ and Claim 4 that
$\{f_k^2\}_{k=1}^{j_2^{\ast}}$ changes signs at most once, where  $1\le j_2^{\ast}\le m_2$ is defined in \eqref{3.13}.
Then using arguments similar to those in the proof of case (A$_3$), we can obtain a contraction by Lemma \ref{complex} (iii).

(iii) For this case, $f^3_{m_3}>0$. It follows from Claim 4 that $g_{m_3}^3>0$ and
\eqref{2.36} holds.
Then, by \eqref{steady-c} and \eqref{steady-d}, we have
\begin{equation}\label{2.37}
	f^1_1+f^2_1>0\;\;\text{and}\;\; g^1_1+g^2_1>0.
\end{equation}
This combined with Lemma \ref{priori} (ii) implies that
$f^1_1,g_1^1>0$ or $f^2_1,g^2_1>0$. Without loss of generality, we assume that $f^1_1,g_1^1>0$. Then
$u^3_{m_3}>u^1_1$ and $v^3_{m_3}>v^1_1$. This combined with \eqref{2.36} yields
 $u^1_1+v^1_1<r$, which contradicts Claim 5. This completes the proof.

\section{Global dynamics}
Consider the following
single species model
\begin{equation}\label{pat-single}
	\begin{cases}
		\ds\frac{{\rm d} u^i_k}{{\rm d}t}=d_1u^{i}_{k-1}-(2d_1+q_1)u^i_k+(d_1+q_1)u^{i}_{k+1}+u^i_k\left(r-u^i_k\right), &(i,k)\in P_o,\\
		\ds\frac{{\rm d} u^i_k}{{\rm d}t}=-(d_1+q_1)u^i_k+d_1u^{i}_{k-1}+u^i_k\left(r-u^i_k\right), &(i,k)\in P_u,\\
		\ds\frac{{\rm d} u^i_k}{{\rm d}t}=-d_1u^i_k+(d_1+q_1)u^{i}_{k+1}+u^i_k\left(r-u^i_k\right), &(i,k)\in P_d,\\
		\ds\frac{{\rm d} u^i_k}{{\rm d}t}=d_1u^{i}_{k-1}-(3d_1+q_1)u^i_{k}+(d_1+q_1)(u^{1}_1+u^2_1)+u^i_k\left(r-u^i_k\right),&(i,k)\in P_r,\\
		\bm u(0)=\bm u_0\ge(\not\equiv)\bm0.
	\end{cases}
\end{equation}
It follows from \cite{54-Cosner-1996,55-Li-Shuai-2010,56-Lu-Takeuchi-1993} that if the
trivial equilibrium $\bm 0$ of \eqref{pat-single} is unstable, then  model \eqref{pat-single}  admits a unique positive equilibrium, which is globally asymptotically stable.
Using similar arguments as in the proof of  \cite[Lemma 2]{37-Chen-Liu-Wu-2022}, we observe that the trivial equilibrium $\bm 0$ of \eqref{pat-single} is unstable. Consequently, model \eqref{pat-single} admits a unique positive equilibrium $\bm {\tilde u}\gg\bm 0$. As a result, model
\eqref{pat-r} has two semi-trivial equilibria $(\bm {\tilde u},\bm 0)$ and $(\bm 0,\bm {\tilde v})$, where
\begin{equation*}
\bm {\tilde u}=(\tilde u^1_{1},\cdots,\tilde u^1_{m_1},\tilde u_{1}^2,\cdots,\tilde u^2_{m_2},\tilde u^{3}_1,\cdots,\tilde u^3_{m_3})\gg\bm 0,
\end{equation*}
and
\begin{equation*}
\bm {\tilde v}=(\tilde v^1_{1},\cdots,\tilde v^1_{m_1},\tilde v_{1}^2,\cdots,\tilde v^2_{m_2},\tilde v^{3}_1,\cdots,\tilde v^3_{m_3})\gg\bm 0.
\end{equation*}

In this section, we investigate the global dynamics of model \eqref{pat-r}. Section \ref{sub21} is devoted to some properties of
semi-trivial equilibria, and their stability is studied in Section \ref{sub2.2}. In Section \ref{sub23}, we show that
competition exclusion occurs under certain condition, i.e., one of the semi-trivial equilibria is globally asymptotically stable.

\subsection{Properties of semi-trivial equilibria}\label{sub21}

Similar to \eqref{fff}, we define
\begin{equation}\label{4set}
\tilde u^1_0=\tilde u^2_0:=\tilde u^3_{m_3}, 
\end{equation}
and denote $\{\tilde f^i_k\}_{(i,k)\in P^*}$ as follows:
\begin{equation}\label{tildef}
\begin{split}
&\tilde f^i_k=\begin{cases}
d_1\tilde u_{k-1}^i-(d_1+q_1)\tilde u_{k}^{i}, &(i,k)\in P_o\cup P_u\cup P_r,\\
0,&(i,k)\in P_d\cup\{(1,m_1+1),(2,m_2+1)\},
\end{cases}
\end{split}
\end{equation}
where $P^*$ is defined in \eqref{pstar}.
Then
\begin{subequations}\label{uf}
\begin{align}
&\tilde f^i_k-\tilde f^{i}_{k+1}+\tilde u^i_k\left(r-\tilde u^i_k\right)=0,\;\;\;\;\;\;\;\;(i,k)\in P_o\cup P_u\cup P_d,\label{uf-a}\\
&\tilde f^i_{k}-\tilde f_1^1-\tilde f^2_1+\tilde u^i_{k}\left(r-\tilde u^i_{k}\right)=0,\;\;\;(i,k)\in P_r.\label{uf-b}
\end{align}
\end{subequations}
The following two results describe the properties of $\bm {\tilde u}$ and $\{\tilde f^i_k\}_{(i,k)\in P^\ast}$.
\begin{lemma}\label{fb}
Suppose that $d_1,q_1>0$, and let $\{\tilde f^i_k\}_{(i,k)\in P^*}$ be defined in \eqref{tildef}. Then the following two statements hold:
\begin{enumerate}
\item [{\rm (i)}] For each $i=1,2$, if there exists $1\le l\le m_i$ such that  $\tilde f_l^i\ge 0$, then
$\tilde u^i_l\ge r$;
\item [{\rm (ii)}]  For each $i=1,2,3$, if there exists $1\le l< m_i$ such that $\tilde{f}^i_{l}\le 0$, then $\tilde{f}^i_{k}<0$ for $k=l+1,\cdots,m_i$.
\end{enumerate}
\end{lemma}
\begin{proof}
(i) Suppose to the contrary that $r>\tilde u^i_l$. Note that $(i,k)\in P_o\cup P_u$ for $k=l,\cdots,m_i$ and $i=1,2$. Then by
 \eqref{uf-a} and induction,
\begin{equation}\label{scont}
\tilde f^i_{k}>0\;\;\text{for}\;\;k=l+1,\cdots,m_i+1,
\end{equation}
which contradicts  $\tilde f^i_{m_i+1}=0$ (see \eqref{tildef}). Therefore, (i) holds.

(ii) By contradiction,
$$k^*:=\min\{l+1\le k\le m_i: \tilde f^i_k\ge0\}$$
is well-defined with $l+1\le k^*\le m_i$.
Then we obtain a contradiction for each of $i=1,2,3$.

If $i=1,2$, then $(i,k)\in P_o\cup P_u$ for $l\le k\le m_i$.
By \eqref{uf-a} and induction,  we have
\begin{subequations}\label{fies}
\begin{align}
&r>\tilde u^i_k\;\;\text{for}\;\;k=k^*,\cdots,m_i,\label{fies-a}\\
&\tilde f^i_{k}>0\;\;\text{for}\;\;k=k^*+1,\cdots,m_i+1,\label{fies-b}
\end{align}
\end{subequations}
which also contradicts  $\tilde f^i_{m_i+1}=0$.

If $i=3$, then $(i,k)\in P_o$ for $l\le k<m_i$.  By \eqref{uf-a} and induction again, we see that \eqref{fies} holds except for $k=m_i+1$, which yields
\begin{equation}\label{fu}
\tilde f^3_{m_3}>0\;\;\text{and}\;\; r>\tilde u_{m_3}^3.
\end{equation}
Since $(3,m_3)\in P_r$, it follows from
\eqref{uf-b} and \eqref{fu} that at least one of $\tilde f_1^1$ and $\tilde f^2_1$ is positive.
Without loss of generality, we assume that $\tilde f_1^1>0$. Then, by the definition of $\tilde f_1^1$ and \eqref{fu}, we see that $r>\tilde{u}^{1}_1$, which contradicts (i).
This completes the proof.
\end{proof}
\begin{lemma}\label{u}
Suppose that $d_1,q_1>0$, and let $\{\tilde f^i_k\}_{(i,k)\in P^*}$ be defined in \eqref{tildef}. Then
the following two statements hold:
\begin{enumerate}
\item [{\rm (i)}] $\tilde f^i_k<0$ for $(i,k)\in M:= \{(i,k):(i,k)\in P_o\cup P_r\cup P_u,\;i=2,3\}$.
That is,
\begin{subequations}\label{pp1}
\begin{align}
&\tilde f^3_k<0 \;\;\text{for}\;\; k=2,\cdots,m_3,\label{pp1-a}\\
&\tilde f^2_k<0 \;\;\text{for}\;\; k=1,\cdots,m_2.\label{pp1-b}
\end{align}
\end{subequations}
\item [{\rm (ii)}]
$\tilde f^1_k+\tilde f^2_k<0$  and $\tilde f^1_k\tilde u^1_k+\tilde f^2_k\tilde u^2_k<0$ for $k=1,\cdots,m_1$.
\end{enumerate}
\end{lemma}
\begin{proof}
(i) Since $(3,1)\in P_d$, it follows from \eqref{tildef} that $\tilde f^3_1=0$. This combined with Lemma \ref{fb} (ii) implies that
\eqref{pp1-a} holds.
Then we show that
\eqref{pp1-b} holds. By Lemma \ref{fb} (ii) again, it suffices to
prove $\tilde f_1^2<0$. Suppose to the contrary that
$\tilde f_1^2\ge0$. Then we claim that $\tilde f_1^1<0$. If the claim is not true, we see from \eqref{uf-b} and \eqref{pp1-a} that
$ r>\tilde{u}^3_{m_3}>\tilde{u}^1_1,\tilde{u}^2_1$, which contradicts Lemma \ref{fb} (i). Thus, $\tilde f_1^1<0$.
Since $\tilde f_1^2\ge0$ and $\tilde f_1^1<0$, it follows that $(\tilde{u}^1_1,\cdots,\tilde{u}^1_{m_1})^T$ is a  lower solution of \eqref{App6} (see Appendix) with $n=m_1$, and $(\tilde{u}^2_1,\cdots,\tilde{u}^2_{m_2})^T$ is an upper solution of \eqref{App6} with $n=m_2$. This combined with Lemma \ref{auxilemma} implies that
\begin{equation}\label{prolast}
(\tilde{u}^1_1,\cdots,\tilde{u}^1_{m_1})\le(\tilde{u}^2_1,\cdots,\tilde{u}^2_{m_1}).
\end{equation}
By $\tilde f_1^2\ge0$, $\tilde f_1^1<0$ and the definitions of $\tilde f_1^1$ and $\tilde f_1^2$, we see that $\tilde{u}^2_1<\tilde{u}_1^1$, which contradicts \eqref{prolast}.
Therefore, $\tilde f_1^2<0$. This completes the proof for (i).

(ii) We first show that, for $k=1$,
\begin{equation}\label{deriv}
\tilde f^1_k+\tilde f^2_k<0\;\;\text{and}\;\;\tilde f^1_k\tilde u^1_k+\tilde f^2_k\tilde u^2_k<0.
\end{equation}
Suppose to the contrary that $\tilde f^1_1+\tilde f^2_1\geq0$. It follows from \eqref{uf-b} and \eqref{pp1} that
$\tilde f^1_1>0$, $\tilde f^2_1<0$ and $r>\tilde {u}_{m_3}^3>\tilde{u}^1_1$, which contradicts Lemma \ref{fb} (i). Therefore, $\tilde f^1_1+\tilde f^2_1<0$.
Noting that $\tilde f^2_1<0$, we see that $\tilde f^1_1\tilde u^1_1+\tilde f^2_1\tilde u^2_1<0$ if
$\tilde f^1_1<0$.
	If $\tilde f^1_1\geq0$, by the definitions of $\tilde f^1_1$ and $\tilde f^2_1$, we have $\tilde u^1_1<\tilde u^2_1$, which also yields
	$$\tilde f^1_1\tilde u^1_1+\tilde f^2_1\tilde u^2_1\leq\left( \tilde f^1_1+\tilde f^2_1\right)\tilde u^2_1<0.$$
Therefore, \eqref{deriv} holds for $k=1$.
	
By induction, it suffices to show that if \eqref{deriv} holds for $k=1,\cdots,k_0$ with $1\le  k_0<m_1$, then \eqref{deriv} holds for $k=k_0+1$. The following proof is divided into two cases:
\begin{enumerate}
\item [ (c$_1$)] there exists $1\le l\le  k_0$ such that $\tilde{f}^1_l\le 0$;
\item [ (c$_2$)] $\tilde{f}^1_k>0$ for all $1\le k\le k_0$.
\end{enumerate}
For	case (c$_1$), it follows from Lemma \ref{fb} (ii) that
$\tilde{f}^1_k<0$ for $l+1\le k\le m_1$. This combined with \eqref{pp1-b} implies that
\eqref{deriv} holds for $k=k_0+1$.
	
For case (c$_2$), noticing that $\tilde{f}^2_k<0$ for $1\leq k\leq k_0$, we see from the definition of $\{\tilde f^i_k\}_{k=1}^{m_i}\; (i=1,2)$  that
\begin{equation}\label{u1u2}
\tilde{u}^1_k\leq \left(\frac{d_1}{d_1+q_1} \right)^k\tilde{u}^3_{m_3}<\tilde{u}^2_k\;\;\text{for}\;\;k=1,\cdots,k_0.
\end{equation}
By \eqref{uf-a},
\begin{equation*}
\left(\tilde{f}^1_{k_0}+\tilde{f}^2_{k_0}\right)-\left(\tilde{f}^1_{k_0+1}+\tilde{f}^2_{k_0+1}\right)=-\tilde{u}^1_{k_0}(r-\tilde{u}^1_{k_0})-\tilde{u}^2_{k_0}(r-\tilde{u}^2_{k_0}).
\end{equation*}
Suppose to the contrary that $\tilde f^1_{k_0+1}+\tilde f^2_{k_0+1}\geq0$. Then at least one of the two inequalities $\tilde{u}^1_{k_0}<r$ and $\tilde{u}^2_{k_0}<r$ holds. This combined with \eqref{u1u2} yields $\tilde{u}^1_{k_0}<r$. By \eqref{pp1-b}, $\tilde f^2_{k_0+1}<0$. This combined with $\tilde f^1_{k_0+1}+\tilde f^2_{k_0+1}\geq0$ yields $\tilde f^1_{k_0+1}>0$, and consequently, $r>\tilde{u}^1_{k_0}>\tilde{u}^1_{k_0+1}$, which contradicts Lemma \ref{fb} (i). Thus, 	$\tilde f^1_{k_0+1}+\tilde f^2_{k_0+1}<0$.

Finally, we show that $\tilde f^1_{k_0+1}\tilde u^1_{k_0+1}+\tilde f^2_{k_0+1}\tilde u^2_{k_0+1}<0$.
Clearly, this holds if $\tilde{f}^1_{k_0+1}<0$.
If $\tilde f^1_{k_0+1}\geq0$, then using similar arguments as in the proof of \eqref{u1u2}, we have $\tilde{u}^1_{k_0+1}<\tilde{u}^2_{k_0+1}$, which yields
$$\tilde f^1_{k_0+1}\tilde u^1_{k_0+1}+\tilde f^2_{k_0+1}\tilde u^2_{k_0+1}\leq\left( \tilde f^1_{k_0+1}+\tilde f^2_{k_0+1}\right)\tilde u^2_{k_0+1}<0.$$
This completes the proof of (ii).
\end{proof}

\subsection{Stability of semi-trivial equilibria}\label{sub2.2}
The stability of the semi-trivial equilibrium $(\bm {\tilde u},\bm 0)$ is determined by the sign of $\lambda_1(d_2,q_2)$:
$(\bm {\tilde u},\bm 0)$ is locally asymptotically stable if $\lambda_1(d_2,q_2)<0$ and unstable if $\lambda_1(d_2,q_2)>0$,
 where
$\lambda_1(d,q)$ is the principal eigenvalue of the following eigenvalue problem:
\begin{equation}\label{2.1}
\begin{cases}
	d\phi^{i}_{k-1}-(2d+q)\phi^i_k+(d+q)\phi^{i}_{k+1}+\left(r-\tilde{u}^i_k\right)\phi^i_k=\lambda \phi^i_k, &(i,k)\in P_o,\\
    -(d+q)\phi^i_k+d\phi^{i}_{k-1}+\left(r-\tilde{u}^i_k\right)\phi^i_k=\lambda \phi^i_k, &(i,k)\in P_u,\\
	-d\phi^i_k+(d+q)\phi^{i}_{k+1}+\left(r-\tilde{u}^i_k\right)\phi^i_k=\lambda \phi^i_k, &(i,k)\in P_d,\\
	d\phi^{i}_{k-1}-(3d+q)\phi^i_{k}+(d+q)(\phi^{1}_1+\phi^2_1)+\left(r-\tilde{u}^i_k\right)\phi^i_k=\lambda \phi^i_k,&(i,k)\in P_r.
	\end{cases}
\end{equation}
We first briefly discuss the existence and uniqueness of the principal eigenvalue $\lambda_1(d,q)$.
Let $A=(a_{ij})$ be a real-valued $m\times m$
square matrix, where $m$ is a positive integer, and let $\sigma(A)$ denote the set of all eigenvalues of $A$. The spectral bound
$s(A)$ of $A$ is defined as
\begin{equation*}
	s(A)=\max\{\text{Re}(\lambda):\lambda\in\sigma(A)\}.
\end{equation*}
It follows from Perron-Frobenius Theorem \cite{60-Li-Schneider-2002} that,
if $A$ is an irreducible essentially nonnegative matrix, then $s(A)$ is an eigenvalue
of $A$. Moreover,  $s(A)$ the unique eigenvalue associated with a positive eigenvector, known as the principal eigenvalue.

Let
$$\bm \phi=(\phi^1_{1},\cdots,\phi^1_{m_1},\phi_{1}^2,\cdots,\phi^2_{m_2},\phi^{3}_1,\cdots,\phi^3_{m_3}),$$
and
$$\bm R=\text{diag}(r-\tilde{u}^1_{1},\cdots,r-\tilde{u}^1_{m_1},r-\tilde{u}_{1}^2,\cdots,r-\tilde{u}^2_{m_2},r-\tilde{u}^{3}_1,\cdots,r-\tilde{u}^3_{m_3}).$$
Then \eqref{2.1} can be can be rewritten as
\begin{equation*}
	(dD+qQ+R)\bm \phi^T=\lambda \bm \phi^T,
\end{equation*}
where $D,Q$ are $m\times m$ square matrix with $m=m_1+m_2+m_3$.
For instance, taking $m_1=1$, $m_2=2$ and $m_3=2$, we have
\begin{equation*}
	D=\left( \begin{matrix}
		-1&0&0&0&1\\
		0&-2&1&0&1\\
		0&1&-1&0&0\\
		0&0&0&-1&1\\
		1&1&0&1&-3
	\end{matrix}\right) ,
	\;
	Q=\left( \begin{matrix}
		-1&0&0&0&0\\
		0&-1&0&0&0\\
		0&0&-1&0&0\\
		0&0&0&0&1\\
		1&1&0&0&-1
	\end{matrix}\right).
\end{equation*}
It is straightforward to verify that $dD+qQ+R$ is  irreducible and essentially nonnegative. This implies the
existence and uniqueness of the principal eigenvalue $\lambda_1(d,q)$.

We now compute the derivative of $\lambda_1(d,q)$ with respect to $q$.
\begin{lemma}\label{lm2.3}
Suppose that $d_1,q_1>0$. Then
\begin{equation}\label{2.2}
\frac{\partial \lambda_1}{\partial q}\left(d_1,q_1 \right)<0.
\end{equation}
\end{lemma}
\begin{proof}
Denote the  eigenvector corresponding to $\lambda_1(d,q)$ by $\bm \phi^T$ with
$$\bm \phi=(\phi^1_{1},\cdots,\phi^1_{m_1},\phi_{1}^2,\cdots,\phi^2_{m_2},\phi^{3}_1,\cdots,\phi^3_{m_3})\gg\bm 0.$$ Substituting $\lambda=\lambda_1(d,q)$ into \eqref{2.1} and taking the derivative with respect to $q$ yield
\begin{equation}\label{2.3}
	\begin{split}
		&\ds\frac{\partial \lambda_1}{\partial q} \phi^i_k+ \lambda_1 \frac{\partial\phi^i_k}{\partial q}\\	
	&=	\begin{cases}
			d\ds\frac{\partial\phi^i_{k-1}}{\partial q}-(2d+q)\ds\frac{\partial\phi^i_k}{\partial q}+(d+q)\frac{\partial\phi^i_{k+1}}{\partial q}-\phi^i_k
			+\phi^i_{k+1}+\left(r-\tilde{u}^i_k\right)\frac{\partial\phi^i_k}{\partial q}, &(i,k)\in P_o,\\
			-(d+q)\ds\frac{\partial\phi^i_k}{\partial q}+d\ds\frac{\partial\phi^i_{k-1}}{\partial q}
			-\phi^i_k+\left(r-\tilde{u}^i_k\right)\frac{\partial\phi^i_k}{\partial q}, &(i,k)\in P_u,\\
			-d\ds\frac{\partial\phi^i_k}{\partial q}+(d+q)\frac{\partial\phi^i_{k+1}}{\partial q}	
			+\phi^{i}_{k+1}+\left(r-\tilde{u}^i_k\right)\frac{\partial\phi^i_k}{\partial q}, &(i,k)\in P_d,\\
			d\ds\frac{\partial\phi^{i}_{k-1}}{\partial q}-(3d+q)\frac{\partial\phi^i_k}{\partial q}+(d+q)\left(\frac{\partial\phi^{1}_1}{\partial q}+\frac{\partial\phi^2_1}{\partial q}\right)\\
			-\phi^i_{k}+(\phi^{1}_1+\phi^2_1)
			+\left(r-\tilde{u}^i_k\right)\ds\frac{\partial\phi^i_k}{\partial q},&(i,k)\in P_r.
		\end{cases}
	\end{split}
\end{equation}	
Multiplying \eqref{2.1} with $\la=\la_1(d,q)$ and \eqref{2.3} by $\dfrac{\partial \phi^i_k}{\partial q}$ and $\phi^i_k$, respectively, and taking the difference, we have
\begin{equation}\label{2.4}
\frac{\partial \lambda_1}{\partial q} \left( \phi^i_k\right) ^2	\\
		=\begin{cases}
		-d\ds\frac{\partial \phi^i_k}{\partial q}\phi^{i}_{k-1}
		-(d+q)\frac{\partial \phi^i_k}{\partial q}\phi^{i}_{k+1}		
		+d\ds\frac{\partial\phi^i_{k-1}}{\partial q}\phi^i_k\\
		+(d+q)\ds\frac{\partial\phi^i_{k+1}}{\partial q}\phi^i_k
		-\left( \phi^i_k\right)^2 +\phi^i_k\phi^i_{k+1}, &(i,k)\in P_o,\\
		-d\ds\frac{\partial \phi^i_k}{\partial q}\phi^{i}_{k-1}
		+d\frac{\partial\phi^i_{k-1}}{\partial q}\phi^i_k-\left( \phi^i_k\right)^2
		, &(i,k)\in P_u,\\
		-(d+q)\ds\frac{\partial \phi^i_k}{\partial q}\phi^{i}_{k+1}
		+(d+q)\frac{\partial\phi^i_{k+1}}{\partial q}\phi^i_k+\phi^i_k\phi^i_{k+1}
		, &(i,k)\in P_d,\\
		-d\ds\frac{\partial \phi^i_k}{\partial q}\phi^{i}_{k-1}
		-(d+q)\frac{\partial \phi^i_k}{\partial q}(\phi^{1}_1+\phi^2_1)		
		+d\frac{\partial\phi^{i}_{k-1}}{\partial q}\phi^i_k\\
		+(d+q)\left(\ds\frac{\partial\phi^{1}_1}{\partial q}+\frac{\partial\phi^2_1}{\partial q}\right)\phi^i_k
		-\left(\phi^i_k \right)^2
		+(\phi^{1}_1+\phi^2_1)\phi^i_k,&(i,k)\in P_r.
	\end{cases}
\end{equation}	
Denote
\begin{equation*}
\alpha_k^i:=\begin{cases}
\left(\ds \frac{d+q}{d}\right)^{m_3+k-1},&i=1,2,\;k=1,\cdots,m_i,\\
\left(\ds \frac{d+q}{d}\right)^{k-1},&i=3,\;k=1,\cdots,m_i.
\end{cases}
\end{equation*}
Multiplying (\ref{2.4}) by $\alpha^i_k$ and summing them over all $(i,k)\in P$, we deduce that
 \begin{equation}\label{2.5}
 	\begin{split}
 	\frac{\partial \lambda_1}{\partial q}\sum_{i=1}^3\sum_{k=1}^{m_i} \left( \phi^i_k\right) ^2=&\sum_{k=2}^{m_3}\left(d\phi^3_{k-1}-(d+q)\phi^3_k \right)\phi^3_{k}\frac{\alpha_{k-1}^3}{d}\\
 	&+\sum_{i=1}^2\left(d\phi^3_{m_3}-(d+q)\phi^i_1 \right)\phi^i_{1}\frac{\alpha_{m_3}^3}{d}\\
 	&+\sum_{i=1}^2\sum_{k=2}^{m_i}\left(d\phi^i_{k-1}-(d+q)\phi^i_k \right)\phi^i_{k}\frac{\alpha_{k-1}^i}{d}.\\
 	\end{split}
 \end{equation}
Since $\lambda_1(d_1,q_1)=0$ and $\bm {\tilde{u}}$ is the positive eigenvector corresponding to $\lambda_1(d_1,q_1)$, we see from \eqref{2.5} that
\begin{equation}\label{2.6}
	\begin{split}
		&\frac{\partial \lambda}{\partial q}(d_1,q_1)\sum_{i=1}^3\sum_{k=1}^{m_i} \left( \tilde{u}^i_k\right) ^2=\sum_{k=2}^{m_3}\tilde{f}^3_k\tilde{u}^3_{k}\frac{(d_1+q_1)^{k-2}}{d_1^{k-1}}+\sum_{i=1}^2\sum_{k=1}^{m_i}\tilde{f}^i_k\tilde{u}^i_{k}\frac{(d_1+q_1)^{m_3+k-2}}{d_1^{m_3+k-1}}.
	\end{split}
\end{equation}
In view of $m_1\le m_2$, the desired result follows from Lemma \ref{u}.
\end{proof}
Then we show that $\lambda_1\left(d_2,q_2\right)\neq0$ for $(d_2,q_2)\in\mathcal S$.
\begin{lemma}\label{lem2}
	Assume that $d_1,q_1>0$ and $(d_2,q_2)\in\mathcal S$ with $\mathcal S$ defined in \eqref{S}. Then
\begin{equation}\label{3.17}
	\lambda_1\left(d_2,q_2\right)\neq0.
\end{equation}	
\end{lemma}
\begin{proof}
It suffices to prove $\lambda_1\left(d_2,q_2\right)\neq0$ for
$(d_2,q_2)\in\mathcal S_1$. Indeed, if $(d_2,q_2)\in \mathcal S_2$, we can derive
$\lambda_1\left(d_2,q_2\right)\neq0$
by interchanging the equations satisfied by $u_k^i$ and $v_k^i$ in model \eqref{pat-r} (see Theorem \ref{priori2} for a detailed explanation).
We now focus on the case where $(d_2,q_2)\in\mathcal S_1$.
Suppose to the contrary that
	\begin{equation}
		\lambda_1\left(d_2,q_2\right)=0
	\end{equation}
and denote the corresponding eigenvector by $\bm \psi^T$ with $$\bm\psi=(\psi^1_{1},\cdots,\psi^1_{m_1},\psi_{1}^2,\cdots,\psi^2_{m_2},\psi^{3}_1,\cdots,\psi^3_{m_3})\gg\bm0.$$ Similar to \eqref{fff}, we also set
\begin{equation}\label{2.311}
	\psi^1_{0}=\psi_{0}^2:=\psi^3_{m_3},
\end{equation}	
and	define $\{\tilde{g}^i_k\}_{(i,k)\in P^\ast}$ as follows:
	\begin{equation}\label{2.7}
		\begin{split}
			&\tilde g^i_k=\begin{cases}
				d_2\psi_{k-1}^i-(d_2+q_2)\psi_{k}^{i}, &(i,k)\in P_o\cup P_u\cup P_r,\\
				0,&(i,k)\in P_d\cup\{(1,m_1+1),(2,m_2+1)\}.
			\end{cases}
		\end{split}
	\end{equation}
	Then $(\tilde{\bm u},\bm \psi) $ is a positive solution of the following system
\begin{subequations}\label{2.8}
	\begin{align}
	&\tilde f^i_k-\tilde f^{i}_{k+1}+\tilde u^i_k\left(r-\tilde u^i_k\right)=0,\;\;\; \;\;\;\;\;(i,k)\in P_o\cup P_u\cup P_d,\\
	&\tilde f^i_k-\tilde f_1^1-\tilde f^2_1+\tilde u_k\left(r-\tilde u^i_k\right)=0,\;\;\;(i,k)\in P_r,\\
&\tilde g^i_k-\tilde g^{i}_{k+1}+\psi^i_k\left(r-\tilde u^i_k\right)=0,\;\;\; \;\;\;\;\;(i,k)\in P_o\cup P_u\cup P_d,\\
	&\tilde g^i_k-\tilde g_1^1-\tilde g^2_1+\psi^i_k\left(r-\tilde u^i_k\right)=0,\;\;\;(i,k)\in P_r.
\end{align}
\end{subequations}	
By replacing $\{f_k^i,g_k^i\}_{(i,k)\in P^\ast}$ in Lemmas \ref{priori}, \ref{canot} and \ref{complex}
with $\{\tilde f_k^i,\tilde g_k^i\}_{(i,k)\in P^\ast}$, we can derive analogous results for the sequences $\{\tilde f_k^i,\tilde g_k^i\}_{(i,k)\in P^\ast}$.

It follows from Lemma \ref{u} (i) that $\tilde f_{m_3}^1<0$ and $\tilde f_{1}^2<0$.
By replacing $\{f_k^i,g_k^i\}_{(i,k)\in P^\ast}$ in the proof of Theorem \ref{priori2} (case (i) $f^3_{m_3}\le0,\;f^2_1\le0$)
with $\{\tilde f_k^i,\tilde g_k^i\}_{(i,k)\in P^\ast}$, we can similarly derive a contradiction. This completes the proof.
\end{proof}

Combining Lemmas \ref{lm2.3} and \ref{lem2}, we can obtain the main result of this subsection.
\begin{theorem}\label{tm2}
	Assume that $d_1,q_1>0$, and let $\mathcal S_1,\mathcal S_2$ be defined in \eqref{S1S2}. Then the following two statements for model \eqref{pat-r} hold:
	\begin{enumerate}
		\item [\rm{(i)}]
		If $(d_2,q_2)\in \mathcal S_1$, then the semi-trivial equilibrium $(\bm 0,\bm {\tilde v})$ is locally asymptotically stable and $(\bm {\tilde u},\bm 0)$ is unstable;
		\item [\rm{(ii)}]
		If $(d_2,q_2)\in \mathcal S_2$, then the semi-trivial equilibrium $(\bm {\tilde u},\bm 0)$ is locally asymptotically stable and
		$(\bm 0,\bm {\tilde v})$ is unstable.
	\end{enumerate}
\end{theorem}
\begin{proof}
	By Lemmas \ref{lm2.3}-\ref{lem2}, we see that $(\bm {\tilde u},\bm 0)$ is locally asymptotically stable for $(d_2,q_2)\in \mathcal S_2$ and unstable for $(d_2,q_2)\in \mathcal S_1$. Since the nonlinear terms of model \eqref{pat-r} are symmetric,
	$(\bm 0,\bm {\tilde v})$ is locally asymptotically stable for $(d_2,q_2)\in \mathcal S_1$ and unstable for $(d_2,q_2)\in \mathcal S_2$. This completes the proof.
\end{proof}

\subsection{Global stability}\label{sub23}
In this subsection, we obtain the global dynamics of model \eqref{pat-r}.
\begin{theorem}\label{tm1}
	Assume that $d_1,q_1>0$, and let $\mathcal S_1,\mathcal S_2$ be defined in \eqref{S1S2}. Then the following two statements for model  \eqref{pat-r} hold:
	\begin{enumerate}
	\item [\rm{(i)}]
	If $(d_2,q_2)\in \mathcal S_1$, then the semi-trivial equilibrium $(\bm 0,\bm {\tilde v})$ is globally asymptotically stable.
	\item [\rm{(ii)}]
	If $(d_2,q_2)\in \mathcal S_2$, then the semi-trivial equilibrium $(\bm {\tilde u},\bm 0)$ is globally asymptotically stable.
	\end{enumerate}
\end{theorem}
\begin{proof}
We only prove (i), and (ii) can be treated similarly.
By Theorems \ref{priori2} and \ref{tm2}, if $(d_2,q_2)\in \mathcal S_1$, then $(\bm 0,\bm {\tilde v})$ is locally asymptotically stable, $(\bm {\tilde u},\bm 0)$ is unstable, and \eqref{pat-r} has no positive equilibrium. Therefore, by the monotone dynamical system theory for competitive systems (\cite[Theorem 1.3]{58-Lam-Munther-2010}, see also \cite{43-Hess-1992,57-Hsu-Smith-Waltman-1996,59-Smith-1995}), $(\bm 0,\bm {\tilde v})$ is globally asymptotically stable. This completes the proof.
\end{proof}

The following result follows from Theorem \ref{tm1} and states that the species with a smaller drift rate will drive the other species to extinction, which suggests that smaller drift rates are favored.
\begin{corollary}\label{4.7}
Suppose that $d_2=d_1>0$. Then the semi-trivial equilibrium $(\bm {\tilde u},\bm 0)$ (resp. $(\bm 0,\bm {\tilde v})$) of model \eqref{pat-r} is
globally asymptotically stable for $q_2>q_1>0$ (resp. $0<q_2<q_1$).
\end{corollary}

\section{Simulations and discussions}\label{sec5}
Our results in Theorem \ref{tm1} are consistent with the results obtained in \cite{37-Chen-Liu-Wu-2022,33-Jiang-Lam-Lou-2021,20-Lou-2006} for model \eqref{App1} in a straight river (i.e., $(D_{kj})$ and $(Q_{kj})$ are defined in \eqref{App2}). For both the straight river and the Y-shaped river network,
as illustrated in Figures \ref{fig2} and \ref{fig1}, respectively,
competition occurs for $(d_2,q_2)\in \mathcal S_1\cup\mathcal S_2$. In particular, when the random dispersal rates of the two species are equal, the species with a smaller drift rate will drive the other species to extinction, which suggests that smaller drift rates are favored.
This result is biologically reasonable: the directed drift washes individuals downstream, which leads to
overcrowding and resource overexploitation at the boundary. Therefore, the species with a larger drift rate faces a higher risk of extinction.

We have shown above that the Y-shaped river network
 illustrated in Figure \ref{fig2} does not influence the global dynamics for two-species competition model.
However, the structure of river networks can influence the distribution of the equilibria. Consider the semi-trivial equilibria $(\bm {\tilde u},\bm 0)$ (resp. $(\bm{w}^\ast,\bm{0})$ for a straight river) as an example.
In the case of a straight river (where $(D_{kj})$ and $(Q_{kj})$ are defined in \eqref{App2}),
Lemma \ref{lem-A4}(ii) (see Appendix) indicates that $\tilde{f}_k:=d_1w^*_{k-1}-(d_1+q_1)w^*_k<0$ for $k=2,\cdots,n$. In contrast, for the Y-shaped river network
shown in Figure \ref{fig2}, Lemma \ref{u} indicates that $\tilde{f}^2_k<0$ for $k=1,\cdots,m_2$ and $\tilde{f}^3_k<0$ for $k=2,\cdots,m_3$.  Interestingly, numerical simulations suggest that some elements in sequence $\{f_k^1\}_{k=1}^{m_1}$ can change sign as the number of patches in river segment $2$ increases, as shown in Figure \ref{fig8}.
Furthermore, in the case of a straight river (where $(D_{kj})$ and $(Q_{kj})$ are defined in \eqref{App2}), Lemma \ref{lem-A4}(i) indicates that
 $w_n^*<w_{n-1}^*<\cdots<w_1^*$, i.e., the distribution of the semi-trivial equilibrium increases as individuals move  downstream. In contrast, for the Y-shaped river network
shown in Figure \ref{fig2}, the distribution of the semi-trivial equilibrium may be non-monotone for both river segments $1,3$ and river segments $2,3$, and
the distribution of the semi-trivial equilibrium at the junction patch may become the second largest, as illustrated in Figure \ref{fig9}.
\begin{figure}[htbp]
	\centering
	\begin{subfigure}[b]{0.48\linewidth}
			\includegraphics[width=\linewidth]{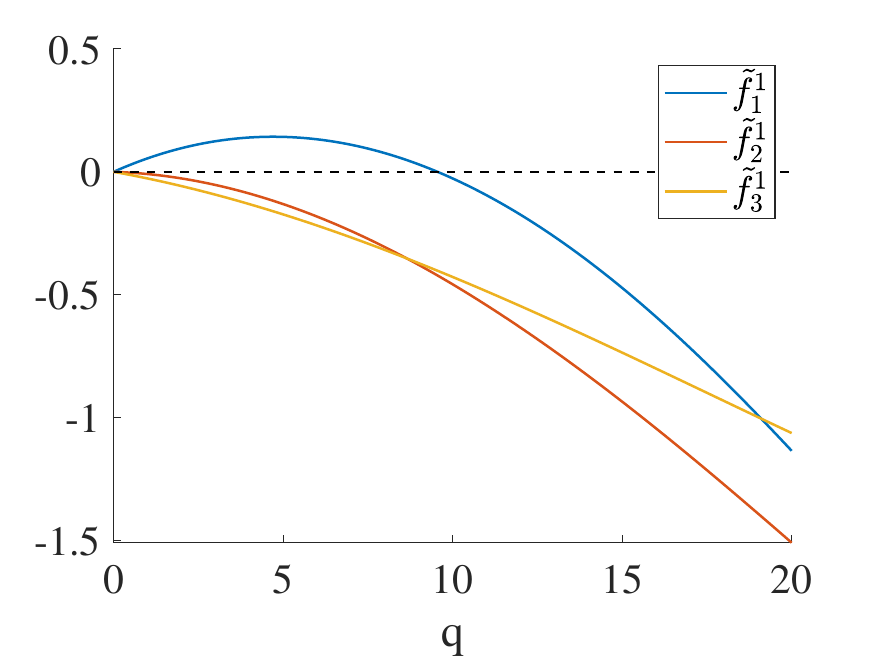}
		\caption{$m_1=3,m_2=10,m_3=4$.}
		\label{fig8-a}
	\end{subfigure}
	\hfill
	\begin{subfigure}[b]{0.48\linewidth}
			\includegraphics[width=\linewidth]{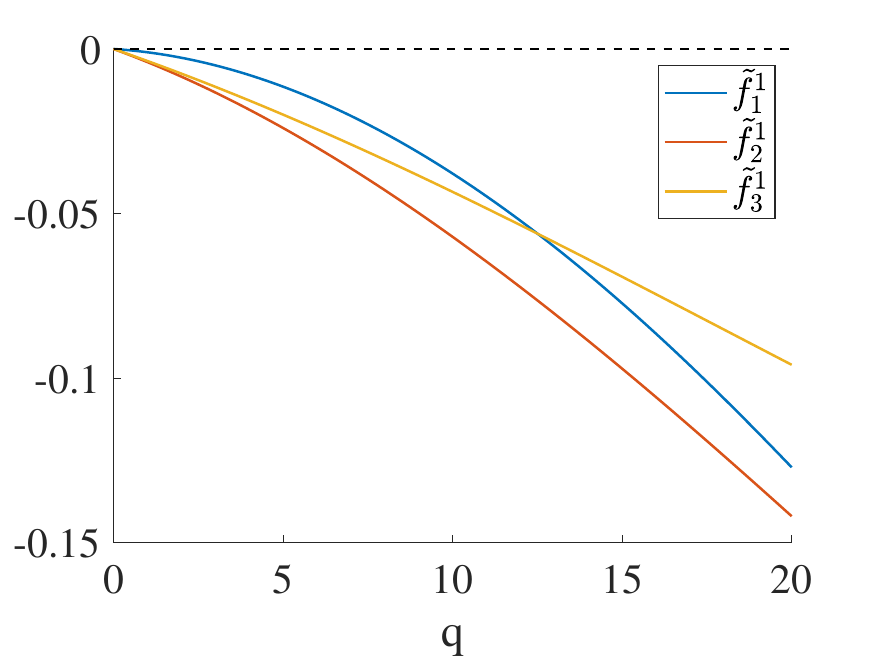}
		\caption{$m_1=3,m_2=7,m_3=4$.}
		\label{fig8-b}
	\end{subfigure}
	\caption{The graphs of $\tilde{f}^1_1$, $\tilde{f}^1_2$ and $\tilde{f}^1_3$ with respect to $q$. Here $d=200$ and $r=3$.}
		\label{fig8}
		\end{figure}
\begin{figure}[htbp]
	\centering
	\begin{subfigure}[b]{0.48\linewidth}
		\includegraphics[width=\linewidth]{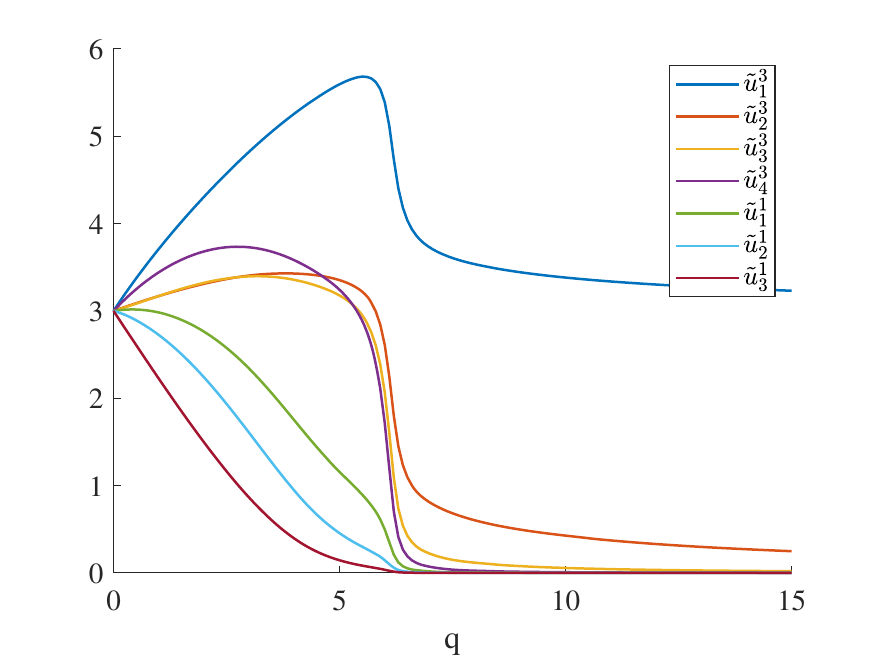}
		\caption{river segments 1,3.}
		\label{fig9-a}
	\end{subfigure}
	\hfill
	\begin{subfigure}[b]{0.48\linewidth}
		\includegraphics[width=\linewidth]{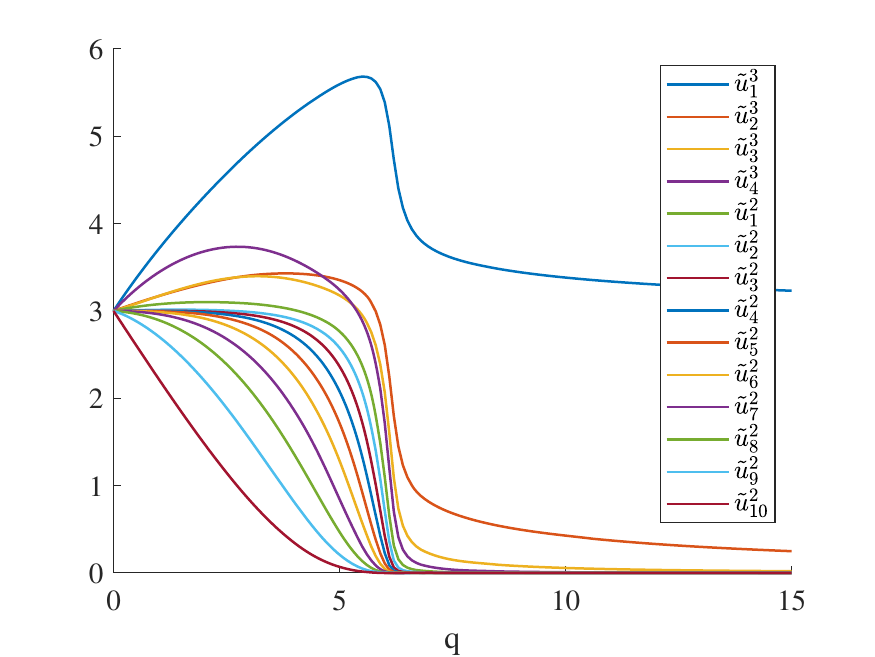}
			\caption{river segments 2,3.}
		\label{fig9-b}
	\end{subfigure}
	\caption{The
		graphs of $\tilde{u}^i_k$ with respect to $q$, where (a) denotes
		the graphs of $\tilde{u}^i_k$ in river segments $1,3$, and (b) denotes the
		graphs of $\tilde{u}^i_k$ in river segments $2,3$. Here $m_1=3$, $m_2=10$, $m_3=4$, $d=1$, and $r=3$. }
	\label{fig9}
\end{figure}

Our results in Theorem \ref{tm1} imply that competition exclusion occurs in model \eqref{pat-r} when $(d_2,q_2)\in\mathcal S_1 \cup\mathcal S_2$, as illustrated in Figure \ref{fig4}. However, and it is
still challenging to explore the dynamics of model \eqref{pat-r} in the blank
regions. For instance, by fixing $q_1=q_2=q>0$, the authors in \cite{32-Jiang-Lam-Lou-2020} conjectured that the semi-trivial equilibrium $(\bm {\tilde u},\bm 0)$ is globally asymptotically stable for $d_1>d_2>0$, while  the semi-trivial equilibrium $(\bm 0,\bm {\tilde v})$ is globally asymptotically stable for $d_2>d_1>0$. To further explore the dynamics of model \eqref{pat-r},  we provide some numerical simulations to illustrate the possible dynamical phenomena that may arise in these blank parameter regions. In Figure \ref{fig5}, we numerically show that
if $(d_2,q_2)\in D_1$, the solution converges to the semi-trivial equilibrium $(\bm 0,\bm {\tilde v})$ as $t\to\infty$;
and if $(d_2,q_2)\in D_2$, and the solution converges to the semi-trivial equilibrium $(\bm {\tilde u},\bm 0)$ as $t\to\infty$. 
Additionally, our numerical simulations suggest that the two species appear to coexist in the regions $(d_2,q_2)\in E_1\cup E_2$, as illustrated in Figure \ref{fig6} for $(d_2,q_2)\in E_1$.

Moreover, it seems both natural and interesting to explore the global dynamics of two-species competition patch models in a spatially heterogeneous environment or under different types of boundary conditions at the downstream.
However, such investigations are clearly non-trivial and present significant challenges. Recently, Vasilyeva et al. \cite{61-Vasilyeva-2024}
studied the Y-shaped river network under a free-flow boundary condition (where there is population loss at downstream end) in the framework of reaction-diffusion models. They showed that the shape of the positive steady
state in a single population model depends on the geometry of the network. For the two-species competition model, they showed that the geometry of the network determines whether higher or intermediate dispersal is favored. Exploring the global dynamics for such reaction-diffusion models remains a challenging and open problem.

\begin{figure}[htbp]
	\centering\includegraphics[width=15cm]{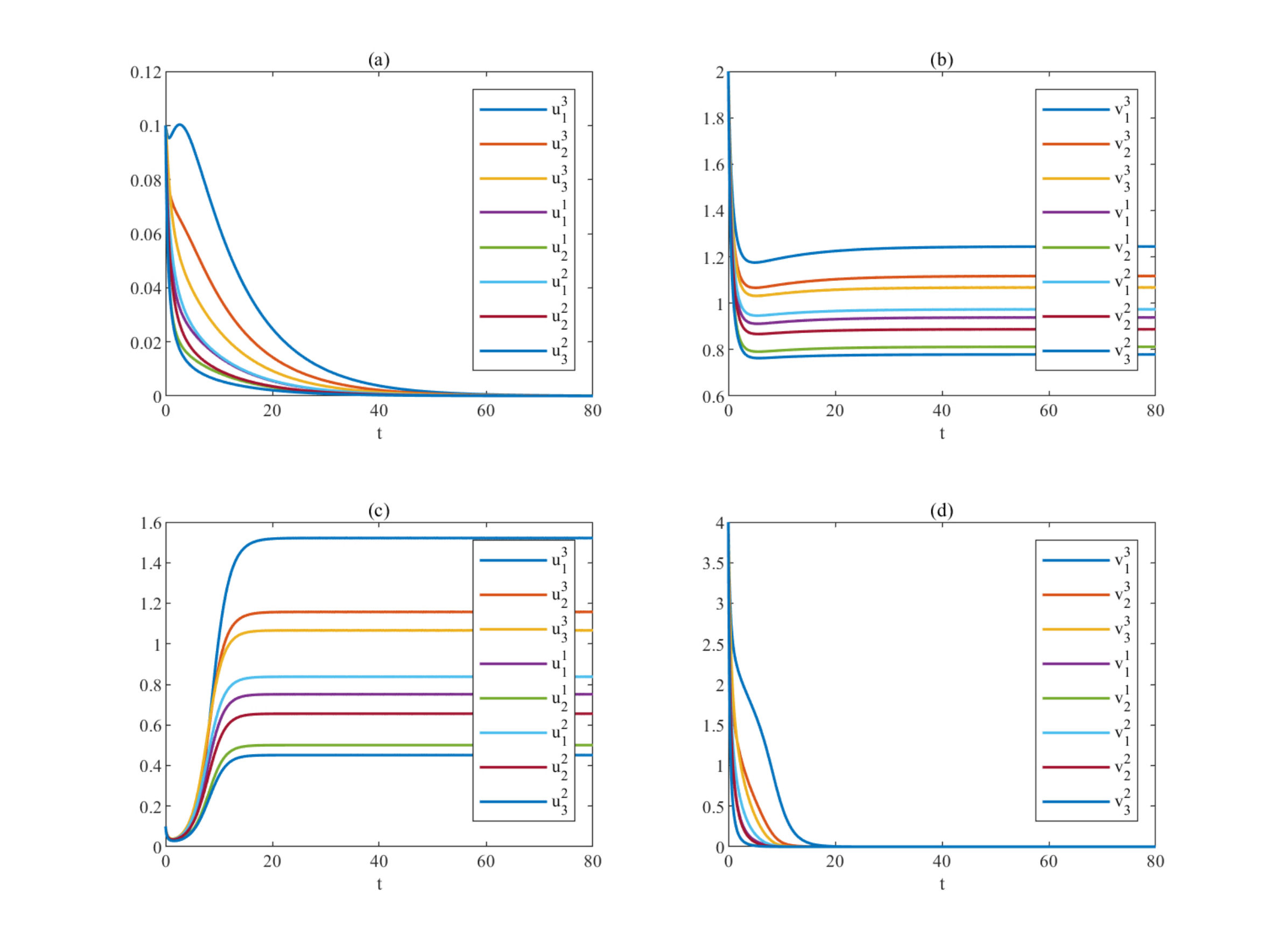}
	\caption{Solutions of  model \eqref{pat-r} with $m_1=2$, $m_2=3$, $m_3=3$, $r=1$, $d_1=1$, and $q_1=1$. For (a)-(b), $(d_2,q_2)\in D_1$ with $d_2=2$ and $q_2=0.5$, and the solution converges to the semi-trivial equilibrium $(\bm 0,\bm {\tilde v})$ as $t\to\infty$. For (c)-(d), $(d_2,q_2)\in D_2$ with $d_2=0.1$ and $q_2=2$, and the solution converges to the semi-trivial equilibrium $(\bm {\tilde u},\bm 0)$ as $t\to\infty$. } \label{fig5}
\end{figure}
\begin{figure}[htbp]
	\centering\includegraphics[width=16cm]{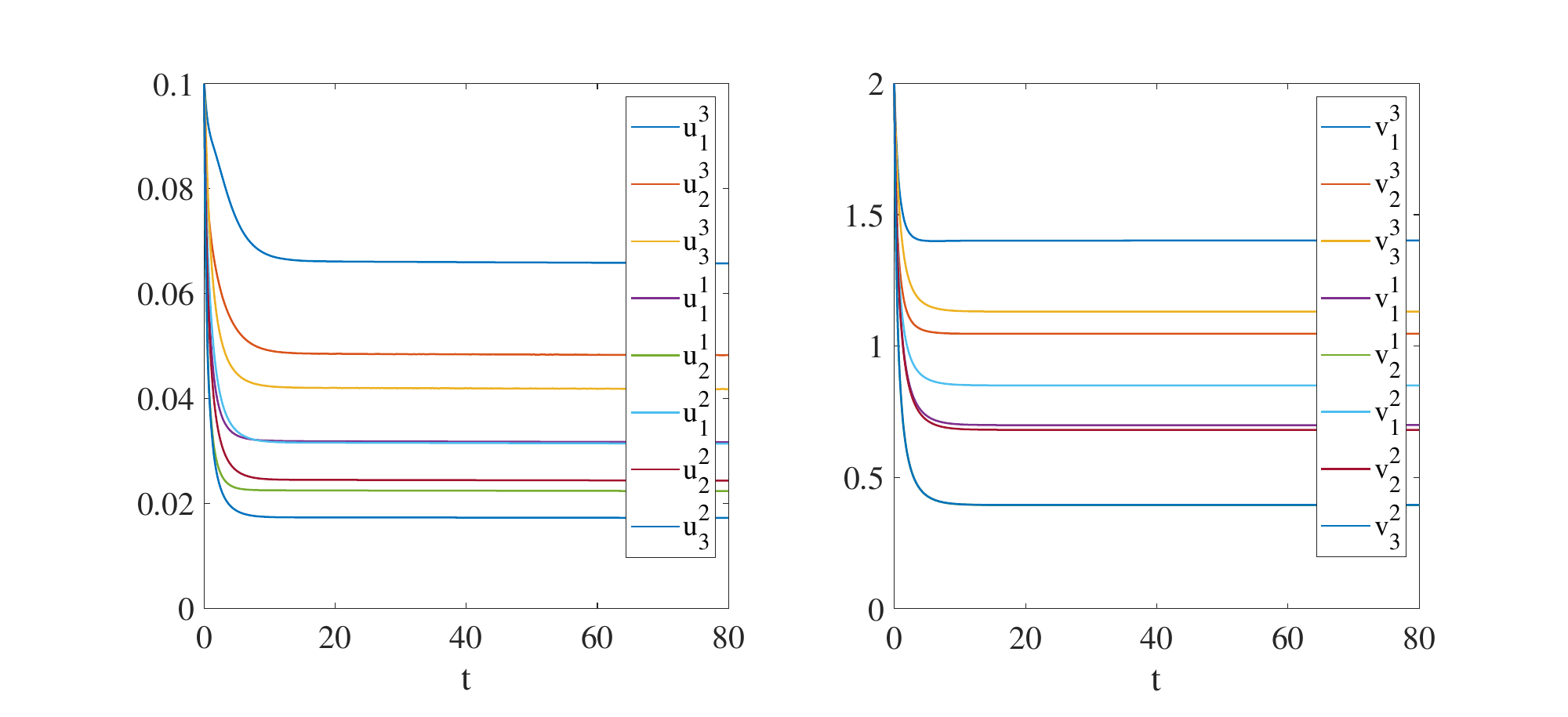}
	\caption{Solutions of  model \eqref{pat-r} with $m_1=2$, $m_2=3$, $m_3=3$, $r=1$, $d_1=1$, $d_2=0.1$, $q_1=1$, and $q_2=0.66$, Here $(d_2,q_2)\in E_1$, and the two species seem to coexist.} \label{fig6}
\end{figure}

\vspace{0.1in}
\begin{appendix}
\setcounter{equation}{0} \setcounter{theorem}{0}
\renewcommand{\thesection}{\Alph{section}}
\renewcommand{\theequation}{\arabic{equation}}
\setcounter{section}{1}
\noindent\textbf{Appendix}\vspace{10pt}

In this part, we revisit some results obtained in \cite{37-Chen-Liu-Wu-2022} for model \eqref{App1} with $(D_{kj})$ and $(Q_{kj})$ defined in \eqref{App2}.
By \cite[Lemma 2]{37-Chen-Liu-Wu-2022}, model \eqref{App1} with $(D_{kj})$ and $(Q_{kj})$ defined in \eqref{App2} admits only two semi-trivial equilibria $(\bm{w}^\ast,\bm{0})$ and $(\bm{0},\bm{z^}\ast)$ with $\bm{w}^\ast,\bm{z^}\ast\gg\bm0$. Here $\bm{w}^\ast$ is the unique positive solution of the following equations:
\begin{equation}\label{App6}
	\ds\sum_{j=1}^{n}(d_1D_{kj}+q_1Q_{kj})w_j+w_k(r-w_k)=0,\;\;k=1,\cdots,n,\tag{A.1}
\end{equation}
where $(D_{kj})$ and $(Q_{kj})$ are defined in \eqref{App2}.

Now we cite the following result from \cite[Lemma 6]{37-Chen-Liu-Wu-2022}.
\begin{lemma}\label{lem-A4}
	Assume that $d_1,q_1>0$, and let  $\bm{w^\ast}=(w_1^\ast,\cdots,w_n^\ast)\gg\bm{0}$ be the unique positive solution of \eqref{App6}, where $(D_{kj})$ and $(Q_{kj})$ are defined in \eqref{App2}. Then
	\begin{enumerate}
		\item [\rm{(i)}]$w_n^*<w_{n-1}^*<\cdots<w_1^*$;
		\item [\rm{(ii)}]
		$d_1w_{k-1}^\ast<(d_1+q_1)w_k^\ast$ for $k=2,\cdots,n$.
	\end{enumerate}
\end{lemma}
By Lemma \ref{lem-A4}, we can obtain the following result.
\begin{lemma}\label{auxilemma}
	Assume that $d_1,q_1>0$, and let $\bm w^{j}=(w_1^j,\cdots,w_n^j)\gg\bm0$ be the unique positive solution of \eqref{App6} with $n=n_j$ for $j=1,2$, where $n_1\le n_2$ and $(D_{kj})$ and $(Q_{kj})$ are defined in \eqref{App2}. Then $w_k^1\le w^2_k$ for $k=1,\cdots,n_1$.
\end{lemma}

\begin{proof}
	It follows from Lemma \ref{lem-A4} that $$d_1w^2_{k-1}-(d_1+q_1)w^2_{k}<0 \;\;\text{for}\;\; k=2,\cdots,n_2.$$
	Then
	$(w^2_{1},\cdots,w^2_{n_1})$ satisfies
	\begin{equation*}
		\begin{cases}
			(d_1+q_1)w^2_2-d_1w^2_1=-w^2_1(r-w^2_1),\\
			d_1w^2_{k-1}-(2d_1+q_1)w^2_{k}+(d_1+q_1)w^2_{k+1}=-w^2_k(r-w^2_k),&k=2,\cdots,n_1-1,\\
			-(d_1+q_1)w^2_{n_1}+d_1w^2_{n_1-1}\le -w^2_{n_1}(r-w^2_{n_1}),
		\end{cases}
	\end{equation*}
	which implies that $(w^2_{1},\cdots,w^2_{n_1})$ is an upper solution of \eqref{App6} with $n=n_1$, and consequently, the desired result holds.
\end{proof}
Define
\begin{equation*}
\begin{split}
&G_1:=\left\{(d,q):d\geq\frac{d_1}{q_1}q,0<q\leq q_1,(d,q)\neq(d_1,q_1) \right\},\\
&G_2:=\left\{(d,q):0<d\leq\frac{d_1}{q_1}q,q\geq q_1,(d,q)\neq(d_1,q_1) \right\}.
\end{split}
\end{equation*}
Obviously, $\mathcal S_1\subset G_1$ and $\mathcal S_2\subset G_2$, where $\mathcal S_1$ and $\mathcal S_2$ are defined in \eqref{S1S2}. Then
 we cite \cite[Lemma 7 and Theorem 4]{37-Chen-Liu-Wu-2022} as follows.
\begin{proposition}\label{lem-A2}
Suppose that  $d_1,q_1>0$ and $(d_2,q_2)\in G_1\cup G_2$. Then model \eqref{App1} with $(D_{kj})$ and $(Q_{kj})$ defined in \eqref{App2} has no positive equilibrium.
\end{proposition}


\begin{proposition}\label{lem-A1}
	Suppose that $d_1,q_1>0$. Then the following two statements for model \eqref{App1} with $(D_{kj})$ and $(Q_{kj})$ defined in \eqref{App2} hold:
	\begin{enumerate}
		\item [\rm{(i)}]
		If $(d_2,q_2)\in G_1$, then $(\bm 0,\bm{z^\ast})$ is globally asymptotically stable
		and $(\bm{w^\ast},\bm{0})$ is unstable;
		\item [\rm{(ii)}]
		If $(d_2,q_2)\in G_2$, then  $(\bm{w^\ast},\bm{0})$ is globally asymptotically stable
		and $(\bm 0,\bm{z^\ast})$ is unstable.
	\end{enumerate}
\end{proposition}
\end{appendix}
\section*{Declarations}

{\bf {Funding:}} S. Chen is supported by National Natural Science Foundation of China (No. 12171117) and Taishan Scholars Program of Shandong Province (No. tsqn 202306137). W. Yan is supported by the Funds for Visiting and Studying of Teachers in Ordinary Undergraduate Universities in Shandong Province.\\
{\bf{Authors' contributions:}} W. Yan and S. Chen wrote the main manuscript text and prepared figures 1-5. All authors reviewed the manuscript.\\
{\bf{Competing Interests:}} The authors declare that they have no conflict of interest.\\
{\bf{Data availability statements:}} All data generated or analyzed during this study is included in this article.\\
{\bf {Ethical Approval:}} This declaration is not applicable.\\

\end{document}